\documentclass[a4paper,11pt]{amsart}
\usepackage[a4paper,margin=28mm]{geometry}
\usepackage[T1]{fontenc}
\usepackage{lmodern}
\usepackage{amsmath,amssymb,mathtools,microtype,enumitem,url}
\usepackage{flafter,float,xurl}
\usepackage{tikz}
\usetikzlibrary{arrows.meta,positioning,calc}
\usepackage[colorlinks=true,linkcolor=blue,citecolor=blue,urlcolor=blue]{hyperref}
\theoremstyle{plain}
\newtheorem{theorem}{Theorem}[section]
\newtheorem{proposition}[theorem]{Proposition}
\newtheorem{lemma}[theorem]{Lemma}
\newtheorem{corollary}[theorem]{Corollary}
\newtheorem*{theoremA}{Theorem A}
\newtheorem*{theoremB}{Theorem B}
\newtheorem*{theoremC}{Theorem C}
\newtheorem*{theoremD}{Theorem D}
\newcommand{\cont}{\mathfrak c}
\newcommand{\covN}{\operatorname{cov}(\mathcal N)}
\newcommand{\Con}{\operatorname{Con}}
\newcommand{\RVM}{\operatorname{RVM}}
\newcommand{\SQp}{\mathrm{SQ}_{p}}
\theoremstyle{definition}

\theoremstyle{remark}
\newtheorem{remark}[theorem]{Remark}
\newcommand{\R}{\mathbb R}
\newcommand{\N}{\mathbb N}
\newcommand{\supp}{\operatorname{supp}}
\newcommand{\Span}{\operatorname{span}}
\newcommand{\one}{\mathbf 1}
\newcommand{\eps}{\varepsilon}
\newcommand{\restr}[1]{{\restriction_{#1}}}
\numberwithin{equation}{section}
\hypersetup{pdftitle={Separable quotients of pointwise function spaces and relative consistency},pdfauthor={Tomasz Kania and Jerzy K\k{a}kol}}
\title[Pointwise quotients and relative consistency]
{Separable quotients of pointwise function spaces\protect\\and relative consistency}
\author[T.~Kania]{Tomasz Kania}
\address[T.~Kania]{Institute of Mathematics\\Czech Academy of Sciences\\\v Zitn\'{a} 25\\115 67 Praha 1\\Czech Republic and Institute of Mathematics\\Faculty of Mathematics and Computer Science\\Jagiellonian University\\{\L}ojasiewicza 6, 30-348 Krak\'{o}w, Poland}
\email{kania@math.cas.cz, tomasz.marcin.kania@gmail.com}
\thanks{RVO: 67985840.}
\author[J.~K\k{a}kol]{Jerzy K\k{a}kol}
\address[J.~K\k{a}kol]{Department of Functional Analysis\\Faculty of Mathematics and Computer Science\\Adam Mickiewicz University\\Uniwersytetu Pozna\'{n}skiego 4\\61-614 Pozna\'{n}, Poland}
\email{kakol@amu.edu.pl}
\date{16 October 2026}
\begin{document}
\raggedbottom
\begin{abstract}
Assuming Jensen's diamond principle, we construct a separable crowded compact zero-dimensional space $K$ of weight $\aleph_1$ and cardinality $2^{\aleph_1}$ such that $C_p(K)$ has no infinite-dimensional Hausdorff separable linear quotient. In particular, $K$ is an Efimov space: it contains neither a nontrivial convergent sequence nor a copy of $\beta\N$. We also prove in ZFC that the existence of a separable quotient of this kind is equivalent to the existence of a metrisable one. On the positive side, established finite-support random-sign methods give a split quotient of $C_p(K)$ onto $(c_0)_p$ whenever $K$ is infinite compact and $w(K)<\covN$, with explicit norm-one operators. Consequently, the universal quotient assertion for compacta of weight $\aleph_1$ is independent relative to the consistency of ZFC alone. A real-valued measurable continuum gives the positive result below the continuum and reduces the remaining problem to separable compacta of weight exactly the continuum. The counterexample uses direct diamond prediction, a simultaneous successor lemma, and a compactness argument converting localised measure limits into differences of point masses.
\end{abstract}
\maketitle

\section{Introduction}

For a Tychonoff space $X$, let $C_p(X)$ be the real vector space of continuous functions on $X$, endowed with pointwise convergence. K\k{a}kol and \'Sliwa \cite{KS2018} asked whether $C_p(K)$ has an infinite-dimensional separable quotient for every infinite compact space $K$. Banakh, K\k{a}kol and \'Sliwa formulated the corresponding metrisable-quotient question in \cite[Problem~2]{BKS2018}. We give a consistent negative answer, and compare it with a positive theorem below the covering number of the null ideal. Together these results settle the relative independence of the assertion restricted to compacta of weight $\aleph_1$.

All compact spaces are Hausdorff, all scalar fields are real, and all linear quotients carry their quotient topology. The term \emph{crowded} means that a space has no isolated points. An \emph{Efimov space} is an infinite compact space containing neither a nontrivial convergent sequence nor a homeomorphic copy of $\beta\N$. Our principal construction is the following.

\begin{theoremA}
Assume $\diamondsuit_{\omega_1}$. There exists a separable crowded compact zero-dimensional space $K$ of weight $\aleph_1$ and cardinality $2^{\aleph_1}$ such that $C_p(K)$ has no infinite-dimensional Hausdorff separable linear quotient, and hence no infinite-dimensional Hausdorff metrisable linear quotient. In particular, $K$ is an Efimov space.
\end{theoremA}

Despite decades of work on Efimov's problem, no example is known in ZFC; all known constructions use additional set-theoretic assumptions (see Hart's survey \cite{Hart}). Theorem~A thus supplies a new consistent Efimov example with a further distinguishing property: its pointwise function space has no infinite-dimensional Hausdorff separable linear quotient. The same compactum therefore gives a consistent negative answer to the separable quotient problem for $C_p(K)$.

Consistent Efimov constructions include those of Fedorchuk \cite{Fedorchuk}, Talagrand \cite{Talagrand}, de la Vega \cite{deLaVega}, Brech \cite{Brech}, Geschke \cite{Geschke}, Dow and Shelah \cite{DowShelah}, and G{\l}odkowski and Widz \cite{GlodkowskiWidz}. Their connections with Grothendieck $C(K)$-spaces and finitely supported measures are studied by K\k{a}kol, Sobota and Zdomskyy \cite{KSZGrothendieck}; related restrictions arising from separation properties of Boolean algebras are studied by Koszmider and Shelah \cite{KoszmiderShelah}.

The absence of a separable quotient does not follow from Efimovness alone: using a construction of de la Vega, K\k{a}kol and \'Sliwa \cite{KS2018} showed under diamond that some Efimov space has such a quotient. In our construction, the absence of metrisable pointwise quotients excludes both nontrivial convergent sequences and copies of $\beta\N$, by \cite[Corollary~3]{BKS2018}. The following result then excludes separable quotients as well; it concerns existence and need not preserve a specified quotient.

\begin{theoremB}
Let $X$ be compact Hausdorff. Then $C_p(X)$ has an infinite-dimensional Hausdorff separable linear quotient if and only if it has an infinite-dimensional Hausdorff metrisable linear quotient.
\end{theoremB}

Put $\cont=2^{\aleph_0}$. The positive result uses $\covN$, the least size of a family of null sets covering the Cantor probability space, equivalently the usual covering number of the Lebesgue null ideal. The subscript $p$ on $c_0$ indicates the topology inherited from $\R^{\N}$.

\begin{theoremC}
Let $K$ be infinite compact Hausdorff. If $w(K)<\covN$, there are continuous linear maps
\[
T:C_p(K)\longrightarrow(c_0)_p,
\qquad S:(c_0)_p\longrightarrow C_p(K),\qquad TS=I.
\]
For the supremum norms, $S:c_0\to C(K)$ is an isometry and $\|T\|=1$. Thus $ST$ is a norm-one projection, continuous for the pointwise topology, onto a copy of $(c_0)_p$.
\end{theoremC}

Theorem~C records a covering-number formulation of established finite-support probabilistic methods. K\k{a}kol, Sobota and Zdomskyy \cite[Section~2, especially Proposition~2.2]{KSZProducts} use finite Bernoulli signs to obtain a finitely supported Josefson--Nissenzweig sequence on every product of two infinite compact spaces. Their product theorem has no weight restriction, and gives complemented copies of $(c_0)_p$ \cite[Theorem~1.2 and Corollary~1.3]{KSZProducts}. Sobota and Zdomskyy \cite[Corollaries~4.7 and~5.1]{SZ} construct such sequences after adding random reals, including a consequence for small Boolean algebras. Here fewer than $\covN$ exceptional null sets are avoided simultaneously. We do not claim a new probabilistic construction or that this exact formulation is a literal statement of those papers. The norm-one splitting also follows from the disjoint-support refinement in \cite[Theorem~1.2]{MSZ}.

Let $\SQp(\aleph_1)$ assert that every infinite compact space $K$ of weight $\aleph_1$ has an infinite-dimensional Hausdorff separable linear quotient of $C_p(K)$.

\begin{theoremD}
If ZFC is consistent, then both $\mathrm{ZFC}+\SQp(\aleph_1)$ and $\mathrm{ZFC}+\neg\SQp(\aleph_1)$ are consistent.
\end{theoremD}

The positive model satisfies Martin's axiom and $\cont=\aleph_2$, so Theorem~C applies; the negative model satisfies diamond, so Theorem~A applies. No measurable-cardinal assumption is needed for this independence statement. If the continuum is real-valued measurable, Theorem~C applies below the continuum; moreover, the closure of every countably infinite subset of a counterexample is a separable compact counterexample of weight exactly $\cont$. Section~\ref{sec:consistency} states the precise additivity hypothesis and separates this stronger consistency assumption from Theorem~D. The unrestricted positive assertion is not proved here.

\subsection*{The construction and the r\^ole of diamond}
The Rosenthal--Lacey theorem already gives an infinite-dimensional separable Banach quotient of $C(K)$ for every infinite compact $K$ \cite{Rosenthal,Lacey}. The difficulty in the pointwise topology is that its continuous functionals are finitely supported measures. A metrisable quotient would give a countably infinite-dimensional relatively weak-star closed subspace $A$ of $L(K)$. Its support is countable, so a sufficiently late coordinate projection is injective and norm preserving on $A$. Diamond predicts the projected generators at such a stage. We arrange that the weak-star closure of $A$ contains a nonzero difference $\delta_y-\delta_x$ killed by that projection, contradicting relative closedness.

Inverse-limit constructions under additional set-theoretic hypotheses go back in particular to Fedorchuk \cite{Fedorchuk}. Talagrand \cite{Talagrand} constructed under CH a compact space whose $C(K)$ has the Grothendieck property. In the present construction, direct diamond prediction replaces the enumeration and anticipation of possible final measure sequences. Lemma~\ref{lem:diamond} and Proposition~\ref{prop:prediction} show that the precise stationary predictor used here is equivalent to diamond. This is the simplification provided by the stronger principle; it does not eliminate the analytic task of preserving earlier limits. Our construction proves the $\ell_1$-Grothendieck property of $C(K)$, and does not establish the full Grothendieck property or all the additional features of Talagrand's compactum.

At a predicted stage we choose disjointly supported signed measures whose positive and negative parts converge to one probability $\lambda$. The successor projection doubles a compact set $F$ of positive $\lambda$-measure and is one-to-one elsewhere. For every clopen $D$ at that stage, a subsequence of the lifted signed measures converges to the difference of the two copies of $\lambda|_{F\cap D}$. The simultaneous successor lemma preserves the limits already arranged, uniformly over all norm-preserving lifts. Countable limits require only diagonal subsequences and extension of coherent measures.

Finally, any finite collection of functions annihilating $A$ factors through a common countable stage. The clopen localisations imply that these finitely many functions agree on a suitable pair of distinct points with the same original projection. The corresponding closed equality sets have the finite intersection property in a fixed compact relation. Compactness therefore gives one pair that works for every function annihilating $A$, and duality puts its difference of point masses in the closure of $A$. This argument dispenses with constructing a final coupling measure on $K^2$.

\subsection*{Relation to the announced Banach-space solution}
The preprint \cite{OpenAI}, dated 23 September 2026, announces relative independence of the Mazur separable quotient problem for Banach spaces. Its small-density sampling argument prompted the comparison with pointwise quotients and the consistency discussion here. The finite-support random-sign method predates that preprint, as the preceding references indicate. We use no theorem from the announced solution and do not undertake a verification of its full construction. Section~\ref{sec:banach} explains why a Banach quotient theorem alone cannot supply the required pointwise quotient.

Section~\ref{sec:prelim} proves the dual criterion and Theorem~B. Section~\ref{sec:positive} proves Theorem~C. Sections~\ref{sec:limits}--\ref{sec:properties} construct the counterexample and prove Theorem~A. Section~\ref{sec:obstruction} explains why one-point splits do not suffice. Sections~\ref{sec:consistency} and~\ref{sec:banach} give the independence result, the real-valued-measurable reduction, and the comparison with Banach spaces.

\section{Duality and preliminary facts}\label{sec:prelim}

We use $\N=\{0,1,2,\ldots\}$. For compact $X$, put $M(X)=C(X)'$, with the variation norm and the weak-star topology $\sigma(M(X),C(X))$, and let $P(X)$ be its Radon probabilities. Write
\[
L(X)=\Span\{\delta_x:x\in X\},\qquad
L(S)=\{\mu\in L(X):\supp\mu\subseteq S\}\quad(S\subseteq X).
\]
The set $S$ need not be closed. The continuous dual of $C_p(X)$ is $L(X)$. For a signed measure $\mu$, write $\mu=\mu^+-\mu^-$ for its Jordan decomposition, $|\mu|=\mu^++\mu^-$ for its total variation, and $\mu|_E$ for restriction to a Borel set. In particular,
$\|\sum_j a_j\delta_{x_j}\|=\sum_j|a_j|$ when the points are distinct. We use regularity of Radon measures and their concentration on their supports. On compact metrisable spaces, finite Borel measures are Radon; nonatomicity is equivalent to vanishing on singletons.

For a continuous map $\pi:X\to Y$, its push-forward satisfies
\[
(\pi_*\mu)(f)=\mu(f\circ\pi)\quad(f\in C(Y)),
\qquad \|\pi_*\mu\|\leqslant\|\mu\|.
\]
Borel push-forwards are also used between compact metrisable spaces. A Borel section $s$ of $\pi$ satisfies $\pi\circ s=\mathrm{id}$; a norm-preserving lift of $v\in M(Y)$ is a measure $\zeta$ with $\pi_*\zeta=v$ and $\|\zeta\|=\|v\|$. A finitely supported Josefson--Nissenzweig sequence consists of norm-one measures in $L(X)$ converging weak-star to zero. Such sequences characterise quotients onto $(c_0)_p$ \cite[Theorem~1]{BKS2019}.

For a linear subspace $A\subseteq M(X)$ and a linear subspace $Z\subseteq C(X)$, define
\[
\begin{split}
A_\perp&=\{f\in C(X):\mu(f)=0\ (\mu\in A)\},\\
Z_M^\perp&=\{\mu\in M(X):\mu(f)=0\ (f\in Z)\},\qquad
Z_L^\perp=Z_M^\perp\cap L(X).
\end{split}
\]
The bipolar theorem for a dual pair gives
\begin{equation}\label{eq:bipolar}
\overline A^{\,w^*}=(A_\perp)_M^\perp,\qquad
\overline A^{\,w^*}\cap L(X)=(A_\perp)_L^\perp\quad(A\subseteq L(X)).
\end{equation}
Here the closure on both left-hand sides is in $M(X)$; intersecting with $L(X)$ gives the relative closure. These are standard duality facts; see \cite[Chapter~8]{Jarchow}. 

The first lemma expresses the metrisable-quotient question in terms of
finitely supported measures.

\begin{lemma}\label{lem:dual}
The space $C_p(X)$ has an infinite-dimensional metrisable Hausdorff quotient if and only if $L(X)$ contains a countably infinite-dimensional relatively weak-star closed linear subspace.
\end{lemma}
\begin{proof}
The topology of $C_p(X)$ is $\sigma(C(X),L(X))$. For a closed linear subspace $Z$, the quotient topology is $\sigma(C(X)/Z,Z_L^\perp)$, and its continuous dual is $Z_L^\perp$; see \cite[Chapter~8]{Jarchow}. A Hausdorff weak topology is metrisable exactly when its continuous dual has at most countable Hamel dimension. Indeed, a countable neighbourhood base uses only countably many functionals, and every continuous linear functional is a finite linear combination of the defining functionals. Conversely, a countable spanning family of the dual defines a countable family of seminorms. For an infinite-dimensional Hausdorff quotient the dual cannot be finite dimensional. Thus a metrisable quotient gives the required relatively closed space $Z_L^\perp$.

Conversely, for such an $A$ put $Z=A_\perp$. Equation~\eqref{eq:bipolar} gives $Z_L^\perp=A$, so the displayed quotient topology is metrisable, Hausdorff and infinite dimensional.
\end{proof}

To apply the successor construction to a subspace of $L(X)$, we need
a sequence of disjointly supported signed measures whose positive and
negative parts have a common weak-star limit.

\begin{lemma}\label{lem:flat}
Let $X$ be compact metrisable and let $A$ be an infinite-dimensional linear subspace of $L(X)$. There exist $\sigma_n\in A$ and $\lambda\in P(X)$ such that the supports of the $\sigma_n$ are pairwise disjoint,
\[
\|\sigma_n\|=2,\qquad \sigma_n(\one)=0,\qquad
\sigma_n\longrightarrow0\quad\hbox{weak-star},
\]
and the Jordan parts $p_n=\sigma_n^+$ and $q_n=\sigma_n^-$ are probabilities converging weak-star to $\lambda$.
\end{lemma}
\begin{proof}
Choose a norm-dense sequence $f_0=\one,f_1,f_2,\ldots$ in $C(X)$. For $x\in X$, the coefficient map $a_x:L(X)\to\R$, $a_x(\mu)=\mu(\{x\})$, is an algebraic linear functional. At step $n$, intersect $A$ with the kernels of $\mu\mapsto\mu(f_j)$ for $j\leqslant n$, and with $\ker a_x$ for the finitely many points used in the earlier supports. The intersection has finite codimension in $A$, so it remains infinite dimensional. Choose a vector of norm $2$ in it. Its support is disjoint from all preceding supports.

For each $j$, the construction gives $\sigma_n(f_j)=0$ whenever $n\geqslant j$.
Given $f\in C(X)$ and $\eps>0$, choose $j$ with
$\|f-f_j\|_\infty<\eps/2$. For $n\geqslant j$,
\[
|\sigma_n(f)|=|\sigma_n(f-f_j)|
\leqslant\|\sigma_n\|\,\|f-f_j\|_\infty<\eps.
\]
Hence $\sigma_n\to0$ weak-star. Since each measure has total mass zero and norm $2$, both Jordan parts have mass one. The space $P(X)$ is weak-star compact and metrisable, so a subsequence of the positive parts converges to some $\lambda\in P(X)$. The differences tend to zero, hence the negative parts have the same limit on this subsequence. Reindex the subsequence to obtain all the stated properties. The measure $\lambda$ may have atoms; no subsequent argument requires it to be nonatomic.
\end{proof}

The next lemma controls all norm-preserving lifts of a finitely supported
measure. The concentration assertion will ensure uniqueness when the
fibres over its support are singletons.

\begin{lemma}\label{lem:lifts}
Let $\pi:X\to Y$ be a continuous surjection of compact Hausdorff spaces, and let $v\in L(Y)$ have norm $1$. Then
\[
\mathcal B(v,\pi)=\{\zeta\in M(X):\|\zeta\|\leqslant1,\ \pi_*\zeta=v\}
\]
is nonempty and weak-star compact. Every $\zeta\in\mathcal B(v,\pi)$
satisfies
\[
\|\zeta\|=1,\qquad |\zeta|\bigl(X\setminus\pi^{-1}(\supp v)\bigr)=0.
\]
For an arbitrary nonzero $v\in L(Y)$, the same conclusions hold with
$1$ replaced by $\|v\|$ in the definition of $\mathcal B(v,\pi)$ and in
the norm equality.
\end{lemma}
\begin{proof}
Choose one point above each atom of $v$ to obtain a norm-preserving lift. Compactness follows from Banach--Alaoglu and weak-star continuity of push-forward. For every $\zeta\in\mathcal B(v,\pi)$,
\[
1=\|v\|\leqslant\|\zeta\|\leqslant1,\qquad
|\pi_*\zeta|\leqslant\pi_*|\zeta|.
\]
The positive measures in the second inequality have equal total mass, so they are equal. Since $|v|$ is concentrated on $\supp v$, the equality $\pi_*|\zeta|=|v|$ gives $|\zeta|(X\setminus\pi^{-1}(\supp v))=0$.
\end{proof}

We recall the inverse-system terminology used in the construction; see
\cite[Sections~2.5 and 3.2]{Engelking} and
\cite[Section~20.5]{KKLPS}, together with the references therein.
An \emph{inverse system} $(K_\alpha,\pi_{\alpha,\beta})$, indexed by an
ordinal interval starting at $\alpha_0$, consists of compact spaces and continuous maps $\pi_{\alpha,\beta}:K_\beta\to K_\alpha$
for $\alpha\leqslant\beta$, with $\pi_{\alpha,\alpha}=\mathrm{id}$ and
$\pi_{\alpha,\gamma}=\pi_{\alpha,\beta}\circ\pi_{\beta,\gamma}$ whenever
$\alpha\leqslant\beta\leqslant\gamma$. Its inverse limit below a limit
ordinal $\rho>\alpha_0$ is the compact space
\[
\varprojlim_{\alpha_0\leqslant\alpha<\rho}K_\alpha
=\left\{(x_\alpha)\in\prod_{\alpha_0\leqslant\alpha<\rho}K_\alpha:
 \pi_{\alpha,\beta}(x_\beta)=x_\alpha\ (\alpha_0\leqslant\alpha\leqslant\beta<\rho)\right\}.
\]
The system is \emph{continuous} if, at each limit stage $\rho>\alpha_0$ in its
index set, the natural map from $K_\rho$ to this inverse limit is a
homeomorphism. In our systems the indices start at $\omega$, the maps
are surjective, and $K_\alpha\subseteq2^\alpha$, where
$2^\alpha=\{0,1\}^\alpha$ has the product topology. The maps are coordinate
restrictions, $\pi_{\alpha,\beta}(x)=x\restr{\alpha}$.

A continuous surjection is a \emph{simple extension} if exactly one fibre
has two points and all other fibres are singletons. In
\cite[Definition~6.1 and Theorem~6.10]{KSZGrothendieck}, K\k{a}kol, Sobota and Zdomskyy
proved that every infinite limit $X$ of a continuous inverse system of
zero-dimensional compact spaces starting with a singleton and using
simple extensions at each successor has the \emph{finitely supported
Josefson--Nissenzweig property}. Section~\ref{sec:obstruction} examines the restriction to
simple extensions and its relation to our construction.

The following factorisation fact allows us to check a continuous function
on the final space at a countable stage.

\begin{lemma}\label{lem:factor}
Let $(K_\alpha,\pi_{\alpha,\beta})_{\omega\leqslant\alpha\leqslant\beta\leqslant\omega_1}$ be a continuous inverse system of compact spaces with surjective bonding maps, where $K_\alpha\subseteq2^\alpha$ and the maps are coordinate projections. Then
\[
C(K_{\omega_1})=\bigcup_{\alpha<\omega_1,\ \alpha\geqslant\omega}
\pi_{\alpha,\omega_1}^*C(K_\alpha),\qquad
\pi_{\alpha,\omega_1}^*g=g\circ\pi_{\alpha,\omega_1}.
\]
\end{lemma}
\begin{proof}
Fix $f\in C(K_{\omega_1})$. For every positive integer $m$, continuity and compactness give a finite cover by relative clopen cylinders on which the oscillation of $f$ is less than $2^{-m}$. The union of the finite coordinate sets used in these covers is countable, and therefore contained in some $\omega\leqslant\alpha<\omega_1$. If $x,y\in K_{\omega_1}$ satisfy
$\pi_{\alpha,\omega_1}(x)=\pi_{\alpha,\omega_1}(y)$, they have the same
coordinates on every finite set used in the covers. For each $m$, any
cylinder in the $m$th cover containing $x$ therefore also contains $y$.
Thus $|f(x)-f(y)|<2^{-m}$ for every $m$, and $f(x)=f(y)$.
We may consequently define
$g(z)=f(x)$ for any $x\in\pi_{\alpha,\omega_1}^{-1}\{z\}$;
this is well defined and $f=g\circ\pi_{\alpha,\omega_1}$. The surjective projection is a quotient map from a compact space onto a Hausdorff space, so $g$ is continuous. Conversely, every displayed pullback is continuous.
\end{proof}

We now prove Theorem~B. The key point is that measures with uniformly
bounded norms and support sizes form weak-star compact sets.

\begin{proof}[Proof of Theorem B]
By the weak-topology description in Lemma~\ref{lem:dual}, a Hausdorff metrisable quotient embeds topologically into $\R^\N$ by evaluation against a countable spanning family of its dual. It is therefore second countable and separable.

Conversely, let $Q=C_p(X)/Z$ be an infinite-dimensional Hausdorff separable quotient and put $A=Z_L^\perp$. This subspace of $L(X)$ is relatively weak-star closed, and the topology of $Q$ is the weak topology determined by $A$. Choose a countable dense set in $Q$ and representatives $f_n\in C(X)$ of its members. These functions separate the elements of $A$: if $\mu\in A$ annihilates every $f_n$, its induced continuous functional on $Q$ vanishes on a dense set and hence is zero. In particular, the map
\[
J:A\longrightarrow\R^\N,\qquad J(\mu)=(\mu(f_n))_{n<\omega},
\]
is injective.

For a positive integer $k$, put
\[
\Delta_k(X)=\{\mu\in L(X):|\supp\mu|\leqslant k,\ \|\mu\|\leqslant k\},
\qquad S_k=A\cap\Delta_k(X).
\]
The set $\Delta_k(X)$ is weak-star compact: it is the image of
\[
X^k\times\{(a_1,\ldots,a_k)\in\R^k:\textstyle\sum_{j=1}^k|a_j|\leqslant k\}
\]
under the continuous map $(x,a)\mapsto\sum_{j=1}^ka_j\delta_{x_j}$.
Since $A$ is relatively closed in $L(X)$ and $\Delta_k(X)\subseteq L(X)$, the set $S_k$ is weak-star compact. The restriction of $J$ to $S_k$ is a continuous injection into the Hausdorff space $\R^\N$, hence a homeomorphism onto its image. Thus every $S_k$ is weak-star metrisable.

Suppose first that some $S_k$ is not compact for the variation norm $\|\mu\|=|\mu|(X)$. The identity map from its compact weak-star topology to its norm topology cannot be continuous, since a continuous image of a compact space is compact. At a point of discontinuity, metrisability gives $\mu_n,\mu\in S_k$ and $\eps>0$ such that
\[
\mu_n\xrightarrow{w^*}\mu,\qquad \|\mu_n-\mu\|\geqslant\eps\quad(n<\omega).
\]
The finitely supported measures $u_n=(\mu_n-\mu)/\|\mu_n-\mu\|$ have norm one and are weak-star null: for every $f\in C(X)$,
$|u_n(f)|\leqslant\eps^{-1}|(\mu_n-\mu)(f)|\to0$.
By \cite[Theorem~1]{BKS2019}, $C_p(X)$ has a quotient isomorphic to $(c_0)_p$, the space $c_0$ with the topology inherited from $\R^\N$. This quotient is infinite dimensional, Hausdorff and metrisable.

Suppose instead that every $S_k$ is norm compact. Choose a countable norm-dense subset $D_k\subseteq S_k$ and let
\[
S=\bigcup_{k\geqslant1}\ \bigcup_{\nu\in D_k}\supp\nu.
\]
The set $S$ is countable. If $\mu\in S_k$ had a nonzero atom at a point $x\notin S$, then
$\|\mu-\nu\|\geqslant|\mu(\{x\})|>0$ for every $\nu\in D_k$, contradicting norm density. Hence every $\mu\in A=\bigcup_{k\geqslant1}S_k$ is supported in $S$. It follows that $A\subseteq L(S)$ has countable Hamel dimension. Lemma~\ref{lem:dual}, including its description of the quotient topology, now shows that the original quotient $Q$ is metrisable. This proves the converse implication.
\end{proof}

\section{Small weight and a split quotient}\label{sec:positive}

For infinite compact $K$, the norm density of $C(K)$ equals $w(K)$. For the upper bound used below, choose $w(K)$ continuous functions separating points, form their rational algebra, and apply Stone--Weierstrass. We first choose the supports for the finite-sign construction.

\begin{lemma}\label{lem:bumps}
Every infinite compact $K$ admits points $(z_i)_{i\geqslant1}$ and functions $h_i\in C(K)$ such that $0\leqslant h_i\leqslant1$, $h_i(z_i)=1$, and the supports of the $h_i$ are pairwise disjoint.
\end{lemma}
\begin{proof}
There is a non-isolated point $p\in K$, since an infinite discrete space cannot be compact. Starting with an open neighbourhood $V_1=K$ of $p$, choose $z_i\in V_i\setminus\{p\}$. By regularity, choose an open $U_i$ containing $z_i$ with $\overline{U_i}\subseteq V_i\setminus\{p\}$, and put $V_{i+1}=V_i\setminus\overline{U_i}$. Every $V_i$ is a neighbourhood of $p$ and contains a point other than $p$. The sets $U_i$ are pairwise disjoint. Regularity and normality give a continuous $h_i:K\to[0,1]$ with $h_i(z_i)=1$ and $\supp h_i\subseteq U_i$.
\end{proof}

\subsection*{The sign estimate and the splitting}

To make the connection with \cite{KSZProducts} precise, write $\Omega_n=\{-1,1\}^n$. On a finite product block their signed measure has the form
\[
\eta_n=\frac1{n2^n}\sum_{s\in\Omega_n}\sum_{i=1}^n s(i)\delta_{(s,i)}.
\]
For $A\subseteq\Omega_n$ and $B\subseteq\{1,\ldots,n\}$, the same second-moment cancellation used below gives
\[
|\eta_n(A\times B)|
\leqslant\frac1n\mathbb E\left|\sum_{i\in B}s(i)\right|
\leqslant\frac{\sqrt{|B|}}n\leqslant\frac1{\sqrt n},
\]
where the expectation is over uniformly distributed $s\in\Omega_n$. This is a short way to see the estimate behind \cite[Proposition~2.2]{KSZProducts}; their argument uses a weak-law estimate. In that construction all sign patterns are incorporated into the product space. Approximation by indicators of clopen rectangles first proves weak-star nullity on a product of Stone spaces; continuous images then give the general product theorem. In the proof below, the signs are auxiliary random parameters attached to points of an arbitrary compact space. We select one sequence of sign patterns that works on a norm-dense family by avoiding fewer than $\covN$ null sets. The difference lies in this simultaneous selection and the hypotheses, not in a new cancellation mechanism.

\begin{proof}[Proof of Theorem~C]
In this proof we index $c_0$ by the positive integers. Put $N_n=4^n$ for $n\geqslant1$. Index the points and functions from Lemma~\ref{lem:bumps} by $z_{n,j},h_{n,j}$, where $1\leqslant j\leqslant N_n$ and $h_{n,j}(z_{n,j})=1$. All bump supports remain pairwise disjoint.

Let $\Omega=\{-1,1\}^{\{(n,j):n\geqslant1,\ 1\leqslant j\leqslant N_n\}}$ with the fair product probability, and let $\eps_{n,j}$ be its coordinate signs. Define
\begin{equation}\label{eq:measures}
\mu_n^\omega=\frac1{N_n}\sum_{j=1}^{N_n}\eps_{n,j}(\omega)\delta_{z_{n,j}}.
\end{equation}
All points in this formula are distinct, so $\|\mu_n^\omega\|=1$. For fixed $f\in C(K)$, independence and the zero means of the signs imply
\begin{align}
\mathbb E|\mu_n^\omega(f)|^2
 &=\frac1{N_n^2}\sum_{j=1}^{N_n}|f(z_{n,j})|^2
 \leqslant\frac{\|f\|_\infty^2}{N_n},\label{eq:secondmoment}\\
\sum_{n=1}^\infty\mathbb E|\mu_n^\omega(f)|
 &\leqslant\|f\|_\infty\sum_{n=1}^\infty2^{-n}<\infty.\label{eq:summable}
\end{align}
By monotone convergence, the Borel set
\[
B_f=\left\{\omega:\sum_n|\mu_n^\omega(f)|=\infty\right\}
\]
is null.

Choose a norm-dense $D\subseteq C(K)$ with $|D|=w(K)<\covN$. The countable product $\Omega$ is a Cantor probability space, so its covering number for null sets is $\covN$. Thus $\bigcup_{f\in D}B_f\ne\Omega$. Fix $\omega$ outside this union and write $\mu_n=\mu_n^\omega$. Then $\mu_n(f)\to0$ for every $f\in D$. Given $f\in C(K)$ and $d\in D$, we have
\[
\limsup_n|\mu_n(f)|\leqslant\|f-d\|_\infty,
\]
because all the measures have norm one. Approximation by $D$ proves convergence to zero on all of $C(K)$.

Set $Tf=(\mu_n(f))_n$. This maps $C(K)$ to $c_0$, has norm at most one, and is continuous from $C_p(K)$ to $(c_0)_p$, since each coordinate is finitely supported. For the fixed signs just selected, set
\[
g_n=\sum_{j=1}^{N_n}\eps_{n,j}(\omega)h_{n,j}.
\]
Disjointness of all bump supports gives
\begin{equation}\label{eq:biorthogonal}
\|g_n\|_\infty=1,\qquad \mu_m(g_n)=\delta_{mn},
\end{equation}
and at each point of $K$ at most one $g_n$ is nonzero. For $a\in c_0$ define
\[
Sa=\sum_{n=1}^\infty a_ng_n.
\]
The series is uniformly convergent: its tails are bounded in norm by $\sup_{n>N}|a_n|$. It therefore defines a continuous function, and evaluation at the selected points shows that $\|Sa\|_\infty=\|a\|_\infty$. At each $x\in K$, the scalar $(Sa)(x)$ is either zero or a fixed multiple of one coordinate $a_n$. This proves continuity of $S:(c_0)_p\to C_p(K)$.

Equation~\eqref{eq:biorthogonal} and norm convergence imply $TS=I$. Since $S$ is a continuous right inverse, $T$ is an open surjection: for any zero-neighbourhood $U\subseteq C_p(K)$, the set $S^{-1}(U)$ is a zero-neighbourhood contained in $T(U)$. Finally, $\|T\|=1$ follows from $TS=I$ and the isometry of $S$, and $ST$ is a nonzero projection of norm one. Its range is topologically isomorphic to $(c_0)_p$. As a subspace of the second-countable space $\R^{\N}$, this range is separable and metrisable.
\end{proof}

\begin{remark}
The norm-one conclusion is not specific to the construction above. If a compact $K$ admits any finitely supported Josefson--Nissenzweig sequence, \cite[Theorem~1.2]{MSZ} supplies such a sequence $(\nu_n)$ with pairwise disjoint supports and with $D=\bigcup_n\supp\nu_n$ discrete in $K$. Enumerate $D$. By regularity and induction, each point $x\in D$ has an open neighbourhood $U_x$ such that the closures of these neighbourhoods are pairwise disjoint: at each step choose a closure meeting $D$ only at the selected point and avoiding the finitely many preceding closures. Choose $h_x\in C(K)$ with $0\leqslant h_x\leqslant1$, $h_x(x)=1$, and $\supp h_x\subseteq U_x$. Writing $\nu_n=\sum_x a_{n,x}\delta_x$, put
\[
g_n=\sum_{x\in\supp\nu_n}\operatorname{sign}(a_{n,x})h_x.
\]
Then $\|g_n\|_\infty=1$, $\nu_m(g_n)=\delta_{mn}$, and the supports of the $g_n$ are pairwise disjoint. The same formulas $Tf=(\nu_n(f))_n$ and $Sa=\sum_n a_ng_n$ therefore give the norm-one splitting. This is a direct consequence of the cited refinement, not an additional novelty of Theorem~C.
\end{remark}

\section{Prediction and preservation at countable limits}\label{sec:limits}

We prepare two ingredients for the construction of $K$. First, diamond
will predict every countable sequence of finitely supported measures
through its coordinate projections. Secondly, we specify the convergence
requirements to be maintained throughout the inverse system and show
that they extend to each countable limit stage.

\subsection{Prediction of measure sequences}\label{subsec:prediction}

We shall apply diamond to sequences of finitely supported measures on $2^{\omega_1}$.  The union of the supports of one such sequence is countable.  The next observation ensures that sufficiently long coordinate projections preserve all its atoms, coefficients and variation norms.

\begin{lemma}\label{lem:support-projection}
Let $S\subseteq2^{\omega_1}$ be countable, and let
$p_\alpha:2^{\omega_1}\to2^\alpha$ be restriction of coordinates.
There is $\eta<\omega_1$ such that, whenever
$\eta\leqslant\alpha<\omega_1$, the map $p_\alpha$ is injective on $S$.
For every such $\alpha$, the map
\[
 (p_\alpha)_*:L(S)\longrightarrow L(p_\alpha[S]),
 \qquad \sum_{j<r}a_j\delta_{x_j}
       \longmapsto\sum_{j<r}a_j\delta_{p_\alpha(x_j)},
\]
is a linear isometric bijection, where $L(S)$ denotes the finitely
supported measures whose supports are contained in $S$; no closedness
of $S$ is required.
\end{lemma}
\begin{proof}
For every pair of distinct points $x,y\in S$, choose
$\xi(x,y)<\omega_1$ with $x(\xi(x,y))\neq y(\xi(x,y))$.
The set of these coordinates is countable.  Choose a countable
$\eta\geqslant\omega$ strictly greater than every chosen coordinate.
Then $p_\alpha(x)\neq p_\alpha(y)$ whenever $\alpha\geqslant\eta$.
Consequently, if $\mu=\sum_{j<r}a_j\delta_{x_j}$ is a reduced finite
representation with $x_j\in S$, then its push-forward has the same
nonzero coefficients at distinct points, and hence
\[
 \|(p_\alpha)_*\mu\|=\sum_{j<r}|a_j|=\|\mu\|.
\]
Every finitely supported measure on $p_\alpha[S]$ has a unique such
lift to $S$, proving the last assertion.
\end{proof}

We recall the set-theoretic terminology used below.  A subset of $\omega_1$ is \emph{club} if it is unbounded and is closed under suprema of increasing countable sequences whose supremum is below $\omega_1$.  A subset is \emph{stationary} if it meets every club.  The intersection of a stationary set with a club is stationary, because the intersection of two clubs is a club.  In particular, every stationary set meets every tail $[\eta,\omega_1)$.

Jensen's principle $\diamondsuit_{\omega_1}$ asserts the existence of a sequence $(d_\alpha)_{\alpha<\omega_1}$ with $d_\alpha\subseteq\alpha$ such that
\[
 \{\alpha<\omega_1:d_\alpha=E\cap\alpha\}
\]
is stationary for every $E\subseteq\omega_1$.  Thus diamond guesses the initial segments of each subset of $\omega_1$ at stationarily many stages.

The following lemma predicts the measure sequences directly. Related
forms of measure prediction appear in \cite[Lemma~6.4]{Glodkowski}
and \cite[Lemma~26]{KoszmiderSilber}. We give the elementary argument
for finitely supported measures.

\begin{lemma}\label{lem:diamond}
Assume $\diamondsuit_{\omega_1}$. There are sequences
$g^\alpha=(g_n^\alpha)_{n<\omega}$ in $L(2^\alpha)$, for
$\omega\leqslant\alpha<\omega_1$, such that, for every sequence
$(v_n)_{n<\omega}$ in $L(2^{\omega_1})$, stationarily many $\alpha$
satisfy
\[
 g_n^\alpha=(p_\alpha)_*v_n\qquad(n<\omega),
\]
and $p_\alpha$ is injective on $\bigcup_n\supp v_n$.
\end{lemma}
\begin{proof}
Fix bijections $e_2:\omega^2\to\omega$ and
$e_3:\omega^3\to\omega$, and enumerate $\mathbb Q$ as
$(q_k)_{k<\omega}$. The fixed points
\[
 C=\{\alpha<\omega_1:\alpha>\omega,\ \omega\cdot\alpha=\alpha\}
\]
form a club: unboundedness follows by countably iterating the map
$\xi\mapsto\omega\cdot\xi$ above any prescribed countable ordinal,
and closedness follows from its continuity at limits.

For $\delta\in C\cup\{\omega_1\}$, consider a labelled array
$\mathcal R=(r_n,a_{n,j},x_{n,j})$, where $r_n<\omega$,
$a_{n,j}\in\mathbb R$ and $x_{n,j}\in2^\delta$ for $j<r_n$.
Empty lists and repeated points are allowed. Represent this array by
\[
\begin{split}
 E_\delta(\mathcal R)={}&
 \{e_3(n,r_n,0):n<\omega\}\\
 &{}\cup\{e_3(n,j,k+1):n,k<\omega,\ j<r_n,\ q_k<a_{n,j}\}\\
 &{}\cup\{\omega\cdot(\xi+1)+e_2(n,j):n<\omega,\ j<r_n,
          \ \xi<\delta,\ x_{n,j}(\xi)=1\}.
\end{split}
\]
Here the last line uses ordinal arithmetic. The first two lines lie
in $\omega$ and recover all list lengths and coefficients, using
the rational lower cuts. Every entry in the last line lies outside
$\omega$, and ordinal division by $\omega$ recovers $\xi+1$ and
the label $(n,j)$. Thus the array is uniquely determined by its
representing subset of $\delta$.

Let $(d_\alpha)_{\alpha<\omega_1}$ witness diamond. If $\alpha\in C$
and $d_\alpha=E_\alpha(\mathcal R)$ for such an array, put
$g_n^\alpha=\sum_{j<r_n}a_{n,j}\delta_{x_{n,j}}$; otherwise let
$g_n^\alpha=0$ for every $n$. These are measures, so repeated atoms
are combined and their coefficients may cancel.

Given $(v_n)_n$, choose finite representations
$v_n=\sum_{j<r_n}a_{n,j}\delta_{x_{n,j}}$ and let $\mathcal R$ be
the resulting array on $2^{\omega_1}$. For $\alpha\in C$, an entry
$\omega\cdot(\xi+1)+m$ lies below $\alpha$ exactly when
$\xi<\alpha$. All entries describing lengths and coefficients also
lie below $\alpha$. Consequently
\[
 E_{\omega_1}(\mathcal R)\cap\alpha
   =E_\alpha(\mathcal R\restriction\alpha),
\]
where restriction leaves lengths and coefficients unchanged and
replaces each point by its first $\alpha$ coordinates. Diamond
therefore gives the asserted equality of measures at stationarily
many stages in $C$. Finally, intersect these stages with the tail
given by Lemma~\ref{lem:support-projection} for
$\bigcup_n\supp v_n$.
\end{proof}

\begin{proposition}\label{prop:prediction}
Suppose that for each $\omega\leqslant\alpha<\omega_1$ there is a sequence $(g_n^\alpha)_{n<\omega}\subseteq L(2^\alpha)$ such that, for every sequence $(v_n)_{n<\omega}\subseteq L(2^{\omega_1})$, stationarily many $\alpha$ satisfy
\[
g_n^\alpha=(p_\alpha)_*v_n\qquad(n<\omega),
\]
where $p_\alpha$ is coordinate restriction. Then $\diamondsuit_{\omega_1}$ holds.
\end{proposition}
\begin{proof}
First fix the family $(g_n^\alpha)$ supplied by the hypothesis. For $\omega\leqslant\alpha<\omega_1$, if $g_0^\alpha=\delta_y$ for some $y\in2^\alpha$, define
\[
d_\alpha=\{\xi<\alpha:y(\xi)=1\};
\]
otherwise put $d_\alpha=\varnothing$. This is well-defined: distinct points of $2^\alpha$ are separated by a coordinate function, so they define distinct Dirac functionals. Put $d_\alpha=\varnothing$ also for $\alpha<\omega$. Thus $d_\alpha\subseteq\alpha$ for every $\alpha$, and the whole sequence has been chosen from the fixed predictor, independently of $E$.

Now fix an arbitrary $E\subseteq\omega_1$ and let $x_E\in2^{\omega_1}$ be its characteristic function. Put $S_E=\{\alpha<\omega_1:d_\alpha=E\cap\alpha\}$. Apply the hypothesis to the constant sequence $v_n=\delta_{x_E}$, $n<\omega$. It says precisely that
\[
T_E=\{\alpha:\omega\leqslant\alpha<\omega_1,\quad
 g_n^\alpha=(p_\alpha)_*\delta_{x_E}\text{ for every }n<\omega\}
\]
is stationary. Coordinate restriction sends a Dirac measure to the Dirac measure at the restricted point: for every $f\in C(2^\alpha)$,
\[
((p_\alpha)_*\delta_{x_E})(f)
=\delta_{x_E}(f\circ p_\alpha)
=f(x_E\restriction\alpha)
=\delta_{x_E\restriction\alpha}(f).
\]
Consequently, for every $\alpha\in T_E$ we have $g_0^\alpha=\delta_{x_E\restriction\alpha}$. The definition of $d_\alpha$ therefore gives
\[
d_\alpha
=\{\xi<\alpha:x_E(\xi)=1\}
=E\cap\alpha.
\]
Hence $T_E\subseteq S_E$.

To see explicitly that $S_E$ is stationary, let $C\subseteq\omega_1$ be any club. Stationarity of $T_E$ gives an $\alpha\in C\cap T_E$. The inclusion just proved implies $\alpha\in C\cap S_E$, so $S_E$ meets every club. The initial choices for $\alpha<\omega$ do not affect this argument, since every member of $T_E$ is at least $\omega$.

Since $E$ was arbitrary and the same fixed sequence $(d_\alpha)$ works for every $E$, this is the definition of a diamond sequence. The stationary set $T_E$ may depend on $E$; no common stationary set for all $E$ is required.
\end{proof}

Together with Lemma~\ref{lem:diamond}, this proves that the exact stationary predictor is equivalent to diamond. CH alone cannot supply it, since CH is consistent with the failure of diamond \cite[Theorem~1.4]{Rinot}. A different CH construction is not excluded. In particular, countable support does not make the atoms of a measure on $2^{\omega_1}$ countable-coordinate objects: each atom may carry $\omega_1$ coordinates.

For a stage space $K_\alpha\subseteq2^\alpha$, the prediction
$g^\alpha$ is \emph{valid} if every $g_n^\alpha$ belongs to
$L(K_\alpha)$ and
$A_\alpha=\Span\{g_n^\alpha:n<\omega\}$ is infinite dimensional.
Otherwise it is \emph{invalid}. In particular, a zero prediction
is invalid.

\subsection{Preservation at countable limits}\label{subsec:countable-limits}

We specify the convergence requirements for the inverse system to be constructed in Section~\ref{sec:recursion} and prove that they are preserved at countable limits.  For subsets $I,J\subseteq\N$, the notation $J\subseteq^*I$ means that $J\setminus I$ is finite.  Convergence along an infinite $J\subseteq\N$ always means convergence as the index tends to infinity through $J$.

At a valid prediction stage $\theta$, let $(D_{\theta,k})_{k<\omega}$
enumerate the clopen subsets of $K_\theta$. The construction introduces
one sequence of measures for each $D_{\theta,k}$.
We index these sequences by $i=(\theta,k)$ and write $b_i=\theta+1$
for their initial stage. All values of $k$ at the same $\theta$ refer
to one predicted subspace; the corresponding localised limits will be
used together in Section~\ref{sec:coupling}.
The sequence with index $i$ is
\[
 (a_{i,n})_{n<\omega}\subseteq L(K_{b_i})
\]
of norm-one measures with pairwise disjoint finite supports.  At every proper stage $\alpha\geqslant b_i$, we retain an infinite set $J_i^\alpha\subseteq\N$ and a measure $\mu_i^\alpha\in M(K_\alpha)$ with $\|\mu_i^\alpha\|\leqslant1$.  Define
\begin{equation}\label{eq:global-lift-sets}
 B_{i,n}^\alpha=
 \{\zeta\in M(K_\alpha):\|\zeta\|\leqslant1,
 (\pi_{b_i,\alpha})_*\zeta=a_{i,n}\}.
\end{equation}
These sets are nonempty and weak-star compact by Lemma~\ref{lem:lifts}.  The three preservation requirements are
\begin{align}
 J_i^\beta&\subseteq^*J_i^\alpha
 &&(b_i\leqslant\alpha\leqslant\beta<\omega_1),
 \label{eq:J-invariant}\\
 (\pi_{\alpha,\beta})_*\mu_i^\beta&=\mu_i^\alpha
 &&(b_i\leqslant\alpha\leqslant\beta<\omega_1),
 \label{eq:mu-invariant}\\
 \sup_{\zeta\in B_{i,n}^\alpha}
 |\zeta(f)-\mu_i^\alpha(f)|&\longrightarrow0
 &&(n\to\infty\text{ along }J_i^\alpha),\quad f\in C(K_\alpha).
 \label{eq:uniform-invariant}
\end{align}
At the initial stage $b_i$, the map $\pi_{b_i,b_i}$ is the identity, so $B_{i,n}^{b_i}=\{a_{i,n}\}$.  Consequently, if $a_{i,n}\xrightarrow{w^*}\mu_i^{b_i}$, the initial requirements hold with $J_i^{b_i}=\N$.

The next lemma allows the construction to continue at every countable limit ordinal.  It produces one infinite set $J_i^\rho$ almost contained in all the earlier sets $J_i^\alpha$, and a measure $\mu_i^\rho$ satisfying the same projection and convergence identities.

\begin{lemma}\label{lem:limit-invariant}
Let $\omega<\rho<\omega_1$ be a limit ordinal, and suppose that the inverse system and requirements \eqref{eq:J-invariant}--\eqref{eq:uniform-invariant} have been constructed below $\rho$.  With $K_\rho$ the inverse limit, all three requirements extend to $\rho$ for every index $i$ with $b_i<\rho$.
\end{lemma}
\begin{proof}
Fix an index $i$ with $b_i<\rho$, and choose an increasing sequence $(\alpha_m)_{m<\omega}$ cofinal in $\rho$, with $\alpha_0\geqslant b_i$.  Each finite intersection $\bigcap_{j\leqslant m}J_i^{\alpha_j}$ is infinite: by \eqref{eq:J-invariant}, it contains all but finitely many members of the infinite set $J_i^{\alpha_m}$.  Inductively choose
\[
 n_m\in\bigcap_{j\leqslant m}J_i^{\alpha_j},
 \qquad n_m>n_{m-1}\quad(m>0),
 \qquad J_i^\rho=\{n_m:m<\omega\}.
\]
For any $\alpha\in[b_i,\rho)$, choose $j$ with $\alpha\leqslant\alpha_j$.  All $n_m$ with $m\geqslant j$ lie in $J_i^{\alpha_j}$, and $J_i^{\alpha_j}\setminus J_i^\alpha$ is finite.  Thus $J_i^\rho\subseteq^*J_i^\alpha$, verifying \eqref{eq:J-invariant}.

Consider the subalgebra
\[
 \mathcal E_i^\rho=
 \bigcup_{b_i\leqslant\alpha<\rho}
 \{g\circ\pi_{\alpha,\rho}:g\in C(K_\alpha)\}
 \subseteq C(K_\rho).
\]
The spaces in this union are increasing with $\alpha$, so the union is a unital algebra.  The projections separate points of the inverse limit; hence this algebra separates points and is uniformly dense by the Stone--Weierstrass theorem.  Set
\[
 \Phi_i(g\circ\pi_{\alpha,\rho})=\mu_i^\alpha(g).
\]
To verify well-definedness, suppose
$g\circ\pi_{\alpha,\rho}=h\circ\pi_{\beta,\rho}$, and choose $\delta\in[\max\{\alpha,\beta\},\rho)$.  The surjectivity of $\pi_{\delta,\rho}$ gives
$g\circ\pi_{\alpha,\delta}=h\circ\pi_{\beta,\delta}$.  Applying $\mu_i^\delta$ and using \eqref{eq:mu-invariant} yields $\mu_i^\alpha(g)=\mu_i^\beta(h)$.  The same passage to a common stage proves linearity.  Surjectivity also gives
\[
 |\Phi_i(g\circ\pi_{\alpha,\rho})|
 \leqslant\|g\|_\infty
 =\|g\circ\pi_{\alpha,\rho}\|_\infty.
\]
Thus $\Phi_i$ extends uniquely to a bounded linear functional of norm at most one on $C(K_\rho)$.  The Riesz representation theorem supplies $\mu_i^\rho\in M(K_\rho)$ with $\|\mu_i^\rho\|\leqslant1$.  For every $\alpha\in[b_i,\rho)$ and $g\in C(K_\alpha)$, its definition gives
\[
 ((\pi_{\alpha,\rho})_*\mu_i^\rho)(g)
 =\mu_i^\rho(g\circ\pi_{\alpha,\rho})
 =\mu_i^\alpha(g).
\]
Equality on continuous functions is equality of Radon measures.  This verifies \eqref{eq:mu-invariant} at $\rho$.

Finally, fix $f\in C(K_\rho)$ and $\eps>0$.  Density supplies $\alpha\in[b_i,\rho)$ and $g\in C(K_\alpha)$ with
$\|f-g\circ\pi_{\alpha,\rho}\|_\infty<\eps$.
For any $\zeta\in B_{i,n}^\rho$, the push-forward
$\eta=(\pi_{\alpha,\rho})_*\zeta$ has norm at most one and satisfies
$(\pi_{b_i,\alpha})_*\eta=a_{i,n}$; hence $\eta\in B_{i,n}^\alpha$.  Therefore
\[
 \sup_{\zeta\in B_{i,n}^\rho}|\zeta(f)-\mu_i^\rho(f)|
 \leqslant 2\eps+
 \sup_{\eta\in B_{i,n}^\alpha}|\eta(g)-\mu_i^\alpha(g)|.
\]
The last term tends to zero along $J_i^\rho$, because it tends to zero along $J_i^\alpha$ and $J_i^\rho\subseteq^*J_i^\alpha$.  Letting $\eps$ decrease to zero proves \eqref{eq:uniform-invariant}.  For each index $i$ with $b_i<\rho$, the construction of $J_i^\rho$ uses only the sets $J_i^\alpha$ for that same index, and the construction of $\mu_i^\rho$ uses only the measures $\mu_i^\alpha$.  None of \eqref{eq:J-invariant}--\eqref{eq:uniform-invariant} imposes a condition relating two different indices.  There are at most countably many such indices, since they are pairs $(\theta,k)$ with $\theta<\rho$ and $k<\omega$.  Applying the preceding argument separately to each of them proves the simultaneous assertion.
\end{proof}

\section{The successor lemma and the inverse system}\label{sec:successor}\label{sec:recursion}

Our aim is to construct, under $\diamondsuit_{\omega_1}$, a continuous
inverse system with $K_\omega=2^\omega$ and $K_\alpha\subseteq2^\alpha$
satisfying \eqref{eq:J-invariant}--\eqref{eq:uniform-invariant}.
The successor lemma below is proved in ZFC. It constructs one extension
that preserves the preceding uniform limits and introduces new
clopen-localised limits.

Diamond is used in Subsection~\ref{subsec:recursion} to obtain the
prediction sequences from Lemma~\ref{lem:diamond}.
At a valid prediction, Lemma~\ref{lem:flat} supplies the signed sequence
to which the successor lemma applies; its extension becomes
$K_{\alpha+1}$ and its limits give \eqref{eq:birth-limit}.
Lemma~\ref{lem:limit-invariant} handles countable limit stages.
The resulting weak-star closure statement is Lemma~\ref{lem:final-closure}.

\subsection{The successor extension}\label{subsec:successor}

\begin{lemma}\label{lem:simultaneous-successor}
Let $K$ be compact, metrisable and zero-dimensional, and let $I$ be a
finite or countably infinite initial segment of $\N$, possibly empty.
For $i\in I$, let $X_i$ be compact Hausdorff, let $r_i:K\to X_i$ be a
continuous surjection, and let $(a_{i,n})_{n<\omega}\subseteq L(X_i)$
have norm one and pairwise disjoint supports. Put
\[
 E_{i,n}=r_i^{-1}(\supp a_{i,n}),\qquad
 B_{i,n}=\{\zeta\in M(K):\|\zeta\|\leqslant1,
                 (r_i)_*\zeta=a_{i,n}\}.
\]
Let $J_i\subseteq\N$ be infinite and suppose that $\mu_i\in M(K)$ satisfies
\begin{equation}\label{eq:old-uniform-hypothesis}
 \sup_{\zeta\in B_{i,n}}|\zeta(f)-\mu_i(f)|\longrightarrow0
 \quad(n\to\infty,\ n\in J_i),\qquad f\in C(K).
\end{equation}
Let $(\sigma_n)_{n<\omega}\subseteq L(K)$ have pairwise disjoint supports
and Jordan decompositions $\sigma_n=p_n-q_n$, where
\[
 p_n,q_n\in P(K),\qquad
 p_n\xrightarrow{w^*}\lambda,\quad q_n\xrightarrow{w^*}\lambda
 \quad\text{for some }\lambda\in P(K).
\]
Thus $\|\sigma_n\|=2$ and $\sigma_n\to0$ weak-star.
Enumerate the clopen subsets of $K$ as $(D_k)_{k<\omega}$, including $K$,
and fix a surjection $\kappa:\N\to\N$ with all fibres infinite.

Then there exist a compact set
$G\subseteq K$ with $\lambda(G)>0$, a set
$F=\supp(\lambda\restr{G})\subseteq\supp\lambda$, and an open set
$U\subseteq K$ such that $\lambda(F)=\lambda(G)>0$,
$U\cap F=\varnothing$ and $\overline U=U\cup F$.
If $\lambda\restr{G}$ is nonatomic, then $F$ is perfect and uncountable;
if $\lambda\restr{G}=c\delta_x$ for some $x\in K$ and $c>0$, then
$F=\{x\}$.
The space
\[
 K'=((K\setminus U)\times\{0\})\cup((U\cup F)\times\{1\})
\]
is compact, metrisable and zero-dimensional. Its projection $\pi:K'\to K$
has the Borel section $s(x)=(x,\one_U(x))$. The following hold.
\begin{enumerate}
\item For each $i\in I$, there is an infinite $J'_i\subseteq J_i$ such that
\begin{equation}\label{eq:successor-old-conclusion}
 \sup_{\substack{\zeta'\in M(K'),\ \|\zeta'\|\leqslant1\\
                  (r_i\circ\pi)_*\zeta'=a_{i,n}}}
 |\zeta'(h)-(s_*\mu_i)(h)|\longrightarrow0
 \quad(n\to\infty,\ n\in J'_i),\qquad h\in C(K').
\end{equation}
Moreover, $\pi_*s_*\mu_i=\mu_i$ and $\|s_*\mu_i\|=\|\mu_i\|$.
\item There are strictly increasing indices $t_m$ such that
$\supp\sigma_{t_m}$ misses $F$. Its unique norm-preserving lift
$\widehat\sigma_{t_m}$ to $K'$ is finitely supported, and
\begin{equation}\label{eq:fresh-localised-limit}
 \frac12\widehat\sigma_{t_m}\xrightarrow{w^*}\nu_k
 \quad(m\to\infty,\ \kappa(m)=k),\qquad
 \nu_k=\frac12\int_{F\cap D_k}
       (\delta_{(x,1)}-\delta_{(x,0)})\,d\lambda(x).
\end{equation}
The measures $\widehat\sigma_{t_m}/2$ have norm one and pairwise disjoint
supports.
\end{enumerate}
\end{lemma}

\begin{figure}[H]
\centering
\begin{tikzpicture}[x=0.9cm,y=0.82cm,font=\small]
\draw[rounded corners] (0,0) rectangle (13,5.9);
\node[anchor=north west] at (0.12,5.8) {$K$};
\filldraw[fill=blue!3,draw=blue!65!black,dashed,rounded corners]
 (4.8,0.55) rectangle (12.5,5.2);
\node[anchor=north east] at (12.35,5.1) {$C_m$};
\filldraw[fill=gray!25,draw=black] (8.6,2.55) rectangle (11.7,3.3);
\node at (10.15,2.93) {$F$};
\draw[fill=blue!12,rounded corners] (5.25,3.45) rectangle (7.9,4.75);
\node at (6.57,4.4) {$U_m\supseteq T_m$};
\node at (5.9,3.86) {$+$};\node at (7.15,3.86) {$+$};
\draw[fill=orange!12,rounded corners] (5.25,0.95) rectangle (7.9,2.25);
\node at (6.57,1.92) {$L_m$};
\node at (5.9,1.35) {$-$};\node at (7.15,1.35) {$+$};
\draw[fill=green!8,rounded corners] (0.6,0.95) rectangle (3.7,2.4);
\node at (2.15,2.05) {$W_{i,m}$};
\node at (2.15,1.4) {$E_{i,n_{i,m}}$};
\draw[fill=blue!12,rounded corners] (0.6,3.45) rectangle (3.7,4.75);
\node at (2.15,4.4) {$P_m$};
\node at (1.5,3.86) {$+$};\node at (2.8,3.86) {$-$};
\node[align=center] at (10.15,1.38)
 {later $C_\ell$ avoid\\$U_m,L_m,W_{i,m}$};
\end{tikzpicture}
\caption{One step of the construction in the proof of Lemma~\ref{lem:simultaneous-successor}. The signs indicate points of $\supp p_{t_m}$ and $\supp q_{t_m}$. The points of $\supp\sigma_{t_m}$ in $P_m\cup T_m$ have second coordinate $1$ in $K'$; all remaining points of this support lie in $L_m$ and have second coordinate $0$. The set $W_{i,m}$ contains $E_{i,n_{i,m}}$ and is disjoint from every later $C_\ell$. Every point of $F$ has two preimages. The figure is schematic: $L_m$ may extend outside $C_m$, and $P_m$ may meet $W_{i,m}$.}
\label{fig:successor}
\end{figure}
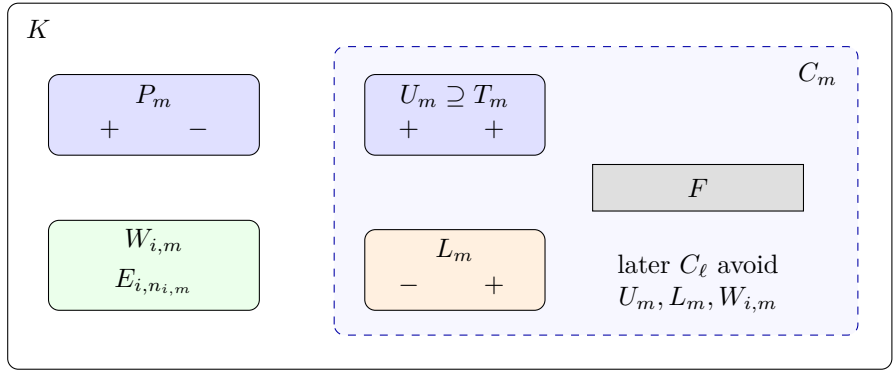

\begin{proof}
By Lemma~\ref{lem:lifts}, every $B_{i,n}$ is nonempty and weak-star
compact, and its members have norm one and are concentrated on
$E_{i,n}$. For fixed $i$, the sets $E_{i,n}$ are pairwise disjoint and compact. Choosing any member of $B_{i,n}$ for $n\in J_i$ and using
\eqref{eq:old-uniform-hypothesis} gives
$|\mu_i(f)|\leqslant\|f\|$ for every $f\in C(K)$; hence
$\|\mu_i\|\leqslant1$.

Add the family $E_{-1,n}=\supp\sigma_n$. For each
$i\in I\cup\{-1\}$, disjointness of the sets $E_{i,n}$ and finiteness
of $\lambda$ imply that their tail unions have measures tending to zero.
Choose $N_i$ with
\[
 \lambda\left(\bigcup_{n\geqslant N_i}E_{i,n}\right)<2^{-i-4}.
\]
Since $\sum_{i=-1}^{\infty}2^{-i-4}=1/4$, the union of all these tail
unions has measure less than $1/4$. By inner regularity, its complement
contains a compact set $G$ of positive measure. Take
$F=\supp(\lambda\restr{G})$. The measure $\lambda\restr{G}$ is
concentrated on its support, so $F\subseteq G$ is compact and
$\lambda\restr{F}=\lambda\restr{G}$. In particular,
\begin{equation}\label{eq:F-properties}
 \lambda(F)>0,\qquad F=\supp(\lambda\restr{F}),\qquad
 F\cap E_{i,n}=\varnothing
 \quad(i\in I\cup\{-1\},\ n\geqslant N_i).
\end{equation}
The inclusion $F\subseteq\supp\lambda$ follows because
$\lambda\restr{G}\leqslant\lambda$. Suppose that $\lambda\restr{G}$
is nonatomic. If $x$ were isolated in $F$, some open set $V$ would satisfy
$V\cap F=\{x\}$. The support condition and concentration on $F$ would
then give
\[
 0<(\lambda\restr{G})(V)=(\lambda\restr{G})(\{x\}),
\]
contrary to nonatomicity. Thus $F$ is perfect. If it were countable,
its closed singletons would express it as a countable union of nowhere
dense sets, contradicting the Baire category theorem. Hence $F$ is
uncountable. If $\lambda\restr{G}=c\delta_x$ for $c>0$, its support
is $\{x\}$. The construction also applies when the restriction has
several atoms or both atomic and nonatomic parts.

Fix a compatible metric $d$ on $K$ and a norm-dense sequence
$(f_j)_{j<\omega}$ in $C(K)$ containing every clopen indicator.
Such a sequence exists because a compact metrisable zero-dimensional
space has only countably many clopen sets: each is a finite union of
members of a fixed countable clopen base.
We recursively choose clopen sets $U_m,L_m,C_m$, clopen neighbourhoods
$W_{i,m}$ for $i\in I$, $i\leqslant m$, and indices $t_m,n_{i,m}$.
Write $P_m=\bigcup_{\ell<m}U_\ell$; thus $P_0=\varnothing$.
All the sets $U_m,L_m,W_{i,m}$ will miss $F$, whereas $C_m$ will be a
clopen neighbourhood of $F$.

At step $m$, for each $i\in I$ with $i\leqslant m$, choose
$n_{i,m}\in J_i$, with $n_{i,m}\geqslant N_i$, larger than the indices
previously chosen for this $i$, and satisfying
\begin{equation}\label{eq:old-rank-estimates}
 \sup_{\zeta\in B_{i,n_{i,m}}}|\zeta(g)-\mu_i(g)|<2^{-m}
 \quad\bigl(g\in\{f_j,f_j\one_{P_m}:j\leqslant m\}\bigr).
\end{equation}
This is possible by \eqref{eq:old-uniform-hypothesis}, since $P_m$ is already fixed and clopen and only finitely many continuous test functions occur.
The compact set $E_{i,n_{i,m}}$ misses $F$ by
\eqref{eq:F-properties}. Zero-dimensionality and compactness give a
clopen neighbourhood $W_{i,m}$ of $E_{i,n_{i,m}}$ disjoint from $F$.

Let $Q_m$ be the union of all previously chosen $U_\ell,L_\ell$
($\ell<m$) and all $W_{i,\ell}$ chosen up to and including step $m$.
This is a finite union of clopen sets disjoint from $F$.
Since $\lambda$ is Radon,
$\lambda(F)=\inf\{\lambda(O):F\subseteq O,\ O\text{ open in }K\}$.
We may therefore choose an open neighbourhood $O_m$ of $F$ with
$\lambda(O_m\setminus F)<2^{-m}$. Choose a clopen neighbourhood $C_m$
of $F$ contained in
\[
 O_m\cap(K\setminus Q_m)\cap\{x:d(x,F)<2^{-m}\},
\]
and also contained in $C_{m-1}$ if $m>0$. For $m=0$ there is no
additional condition involving a previous $C_m$. Each displayed open
set contains the compact set $F$, so the clopen neighbourhood exists
by a finite subcover from a clopen base. We have
\begin{equation}\label{eq:C-rank-estimates}
 C_m\subseteq\{x:d(x,F)<2^{-m}\},\qquad
 \lambda(C_m\setminus F)<2^{-m}.
\end{equation}

Set $D=D_{\kappa(m)}$. Choose $t_m\geqslant N_{-1}$, larger than all
earlier $t_\ell$, so that, for $j\leqslant m$,
\begin{align}
 |\sigma_{t_m}(f_j)|&<2^{-m},\label{eq:fresh-rank1}\\
 |\sigma_{t_m}(f_j\one_{P_m})|&<2^{-m},\label{eq:fresh-rank2}\\
 |p_{t_m}(f_j\one_{C_m\cap D})
       -\lambda(f_j\one_{C_m\cap D})|&<2^{-m}.
       \label{eq:fresh-rank3}
\end{align}
The sets $P_m,C_m,D$ are fixed and clopen, so these finitely many estimates follow simultaneously from $\sigma_n\to0$ and $p_n\to\lambda$ weak-star.

Put $T_m=\supp p_{t_m}\cap C_m\cap D$. This finite set misses $F$.
Choose a clopen $U_m\subseteq C_m\cap D$ containing $T_m$, disjoint from
$F$, and meeting $\supp\sigma_{t_m}$ in exactly $T_m$.
For each point of $T_m$, a clopen neighbourhood inside $C_m\cap D$
can be chosen to avoid the compact set
$F\cup(\supp\sigma_{t_m}\setminus T_m)$; the union of these finitely
many neighbourhoods works. Take $U_m=\varnothing$ when $T_m$ is empty.
Since $C_m$ misses every earlier $U_\ell$, the sets $U_m$ are pairwise
disjoint. The finite set
\[
 \supp\sigma_{t_m}\setminus(P_m\cup T_m)
\]
is disjoint from the compact set $F\cup P_m\cup U_m$. Choose a clopen
neighbourhood $L_m$ of that finite set avoiding this compact set.
This finishes step $m$. The construction ensures that all later
$C_\ell$ and $U_\ell$ miss $L_m$ and every $W_{i,m}$.

Let $U=\bigcup_mU_m$, an open set disjoint from $F$. We prove
$\overline U\setminus U=F$. If $y\in\overline U\setminus U$, choose
$y_j\in U$ with $y_j\to y$ and write $y_j\in U_{m_j}$.
Each finite union of the clopen sets $U_m$ is closed and excludes $y$;
thus $m_j\to\infty$. Equation~\eqref{eq:C-rank-estimates} gives
$d(y_j,F)<2^{-m_j}\to0$, so $y\in F$.
Conversely, let $x\in F$ and let $V$ be a neighbourhood of $x$.
Choose a clopen neighbourhood $V_0$ with $x\in V_0\subseteq V$.
The support condition in \eqref{eq:F-properties} gives
$\lambda(F\cap V_0)>0$. The function $\one_{V_0}$ is one of the $f_j$.
Infinitely many steps have $D_{\kappa(m)}=K$, and at all sufficiently
large such steps \eqref{eq:fresh-rank3} yields
\[
 p_{t_m}(C_m\cap V_0)
 >\lambda(C_m\cap V_0)-2^{-m}
 \geqslant\lambda(F\cap V_0)-2^{-m}>0.
\]
Consequently $T_m\cap V_0\neq\varnothing$, and $U_m$ meets $V$.
This proves $F\subseteq\overline U$ and the asserted identity.

Put
\[
 K'_0=(K\setminus U)\times\{0\},\qquad
 K'_1=(U\cup F)\times\{1\}.
\]
Both sets are closed in $K\times\{0,1\}$, because $K\setminus U$
and $U\cup F=\overline U$ are closed. Thus $K'=K'_0\cup K'_1$ is
compact, metrisable and zero-dimensional. The sets $K'_0$ and $K'_1$
are also clopen in $K'$, since
$K'_e=K'\cap(K\times\{e\})$ for $e=0,1$. The function
\[
 \chi_1:K'\longrightarrow\R,\qquad
 \chi_1(x,e)=e=\one_{K'_1}(x,e),
\]
is therefore continuous. The projection $\pi(x,e)=x$ is a continuous
surjection, with
\[
 \pi^{-1}\{x\}=
 \begin{cases}
 \{(x,1)\},&x\in U,\\
 \{(x,0),(x,1)\},&x\in F,\\
 \{(x,0)\},&x\in K\setminus(U\cup F).
 \end{cases}
\]
As $U$ is Borel, $s(x)=(x,\one_U(x))$ is a Borel map into $K'$
and $\pi\circ s=\mathrm{id}_K$.
For a signed Radon measure $\mu$ on $K$, define
$(s_*\mu)(B)=\mu(s^{-1}(B))$ for every Borel set $B\subseteq K'$.
This finite Borel measure is Radon because $K'$ is compact and
metrisable. For $h\in C(K')$ it satisfies
$(s_*\mu)(h)=\int_K h(s(x))\,d\mu(x)$, and
\[
 \pi_*s_*\mu=\mu,\qquad
 \|\mu\|=\|\pi_*s_*\mu\|\leqslant\|s_*\mu\|
          \leqslant\|\mu\|.
\]
The first equality follows from $\pi\circ s=\mathrm{id}_K$.
For either map $T=\pi$ or $T=s$ and every continuous function $h$
on its target, the integral formula gives
\[
 |(T_*\eta)(h)|=\left|\int h\circ T\,d\eta\right|
 \leqslant\|h\|\,\|\eta\|.
\]
Taking the supremum over $\|h\|\leqslant1$ gives the two norm
inequalities.

We determine the intersections of $U$ with the supports of the measures.
Fix $i\in I$ and $m\geqslant i$. Since $E_{i,n_{i,m}}\subseteq W_{i,m}$
and $C_\ell\cap W_{i,m}=\varnothing$ for every $\ell\geqslant m$,
none of the sets $U_\ell\subseteq C_\ell$ with $\ell\geqslant m$
meets $E_{i,n_{i,m}}$. As $P_m=\bigcup_{\ell<m}U_\ell$, we obtain
\begin{equation}\label{eq:old-selector-identity}
 \one_U\restr{E_{i,n_{i,m}}}
 =\one_{P_m}\restr{E_{i,n_{i,m}}}.
\end{equation}
Next put $S_m=\supp\sigma_{t_m}$. By construction,
$U_m\cap S_m=T_m$ and $S_m\setminus(P_m\cup T_m)\subseteq L_m$.
Since $T_m\subseteq U_m$, this gives
$S_m\subseteq P_m\cup U_m\cup L_m$.
For every $\ell>m$, the choice of $C_\ell$ ensures that
$U_\ell\subseteq C_\ell$ is disjoint from $P_m\cup U_m\cup L_m$,
and hence from $S_m$. Thus only the sets $U_\ell$ with $\ell\leqslant m$
contribute to $U\cap S_m$, and
\[
 U\cap S_m=(P_m\cap S_m)\cup T_m.
\]
The two sets on the right are disjoint, because
$T_m\subseteq C_m$ and $C_m\cap P_m=\varnothing$.
The support $S_m$ misses $F$, so each $x\in S_m$ has the unique
preimage $s(x)$. Writing $\sigma_{t_m}=\sum_{x\in S_m}c_{m,x}\delta_x$,
we obtain the norm-preserving lift
\[
 \widehat\sigma_{t_m}=\sum_{x\in S_m}c_{m,x}\delta_{s(x)},
 \qquad \|\widehat\sigma_{t_m}\|=
 \sum_{x\in S_m}|c_{m,x}|=2.
\]
By Lemma~\ref{lem:lifts}, applied after division by two, every
norm-preserving lift is concentrated on $\pi^{-1}(S_m)$.
Since each of these fibres is a singleton, the push-forward requirement
fixes its coefficient at $s(x)$ to be $c_{m,x}$. This proves uniqueness.
Since $T_m\subseteq\supp p_{t_m}$, the disjointness of the Jordan
supports gives, for every $g\in C(K)$,
\begin{equation}\label{eq:fresh-selector-identity}
 \widehat\sigma_{t_m}((g\circ\pi)\chi_1)
 =\sigma_{t_m}(g\one_{P_m})
       +p_{t_m}(g\one_{C_m\cap D_{\kappa(m)}}).
\end{equation}
Indeed, $\sigma_{t_m}\restr{T_m}=p_{t_m}\restr{T_m}$ and
$p_{t_m}(g\one_{T_m})=p_{t_m}(g\one_{C_m\cap D_{\kappa(m)}})$.
The measures $\widehat\sigma_{t_m}$ have pairwise disjoint supports,
because their supports project onto the pairwise disjoint sets $S_m$.

Fix $i\in I$ and put $J'_i=\{n_{i,m}:m\geqslant i\}$.
The map $m\mapsto n_{i,m}$ is strictly increasing, so convergence
as $m\to\infty$ is the same as convergence as $n\to\infty$ along $J'_i$.
Given $h\in C(K')$, the function $x\mapsto h(x,0)$ is continuous
on the closed set $K\setminus U$, and $x\mapsto h(x,1)$ is continuous
on the closed set $U\cup F$. By the Tietze extension theorem, choose
$h_0,h_1\in C(K)$ such that
\[
 h_0(x)=h(x,0)\quad(x\in K\setminus U),\qquad
 h_1(x)=h(x,1)\quad(x\in U\cup F).
\]
Let $\zeta'$ be any measure occurring in the supremum in
\eqref{eq:successor-old-conclusion} with $n=n_{i,m}$, and set
$\zeta=\pi_*\zeta'$. Then $\zeta\in B_{i,n_{i,m}}$.
Write $E=E_{i,n_{i,m}}$. Lemma~\ref{lem:lifts} for
$r_i\circ\pi$ gives
\[
 |\zeta'|(K'\setminus\pi^{-1}(E))=0,
 \qquad |\zeta|(K\setminus E)=0.
\]
Since $E\cap F=\varnothing$, the map
$\pi:\pi^{-1}(E)\to E$ is a continuous bijection between compact
Hausdorff spaces, hence a homeomorphism, whose inverse is $s\restr{E}$.
For a Borel set $B\subseteq K'$, this implies
\[
 \begin{split}
 \zeta'(B)
 &=\zeta'(B\cap\pi^{-1}(E))\\
 &=\zeta'\bigl(\pi^{-1}(s^{-1}(B))\cap\pi^{-1}(E)\bigr)
 =\zeta(s^{-1}(B)).
 \end{split}
\]
Thus $\zeta'=s_*\zeta$. Moreover, \eqref{eq:old-selector-identity}
says that $s(x)=(x,\one_{P_m}(x))$ for $x\in E$, so
\[
 \zeta'(h)=\zeta\bigl(h_0+(h_1-h_0)\one_{P_m}\bigr).
\]
Choose fixed indices $j_0,j_1$ with $\|h_e-f_{j_e}\|<\eps$ for $e=0,1$.
The function $h_0+(h_1-h_0)\one_{P_m}$ differs by less than $\eps$
from $f_{j_0}+(f_{j_1}-f_{j_0})\one_{P_m}$: the difference is
$h_0-f_{j_0}$ off $P_m$ and $h_1-f_{j_1}$ on $P_m$.
For $m\geqslant\max(i,j_0,j_1)$, the three estimates in
\eqref{eq:old-rank-estimates} and the norm bounds thus give
\[
 \left|\zeta'(h)-
 \mu_i\bigl(h_0+(h_1-h_0)\one_{P_m}\bigr)\right|
 \leqslant2\eps+3\cdot2^{-m},
\]
uniformly over all these $\zeta'$.
Since $P_m$ increases to $U$, dominated convergence with respect to
$|\mu_i|$ gives
\[
 \mu_i\bigl(h_0+(h_1-h_0)\one_{P_m}\bigr)
 \longrightarrow\mu_i\bigl(h_0+(h_1-h_0)\one_U\bigr)
 =(s_*\mu_i)(h).
\]
Letting first $m\to\infty$ and then $\eps\to0$ proves part (1).

For part (2), fix $k$ and restrict to $m\in\kappa^{-1}\{k\}$.
For each fixed $j$ and all $m\geqslant j$ in this set,
\eqref{eq:fresh-rank1} gives
$|\widehat\sigma_{t_m}(f_j\circ\pi)|<2^{-m}$.
Equations \eqref{eq:C-rank-estimates},
\eqref{eq:fresh-rank2}--\eqref{eq:fresh-rank3} and
\eqref{eq:fresh-selector-identity} give the further estimate
\[
 \left|\widehat\sigma_{t_m}((f_j\circ\pi)\chi_1)
       -\int_{F\cap D_k}f_j\,d\lambda\right|
 \leqslant(2+\|f_j\|)2^{-m}.
\]
For $g\in C(K)$ with $\|g-f_j\|<\eps$, the left side with $g$
in place of $f_j$ is at most
$3\eps+(2+\|f_j\|)2^{-m}$: the lifted measure has norm two,
and $g\mapsto\int_{F\cap D_k}g\,d\lambda$ is a continuous linear
functional of norm at most one. Similarly,
$|\widehat\sigma_{t_m}(g\circ\pi)|\leqslant2\eps+2^{-m}$.
These estimates extend both limits to every $g\in C(K)$.
For an arbitrary $h\in C(K')$, choose the extensions $h_0,h_1$
as in the proof of part (1). Then
\[
 h=h_0\circ\pi+((h_1-h_0)\circ\pi)\chi_1.
\]
The two established limits yield
\[
 \widehat\sigma_{t_m}(h)\longrightarrow
 \int_{F\cap D_k}\bigl(h(x,1)-h(x,0)\bigr)\,d\lambda(x).
\]
Division by two proves \eqref{eq:fresh-localised-limit} and completes
the proof.
\end{proof}

\subsection{The inverse system and its final measures}\label{subsec:recursion}

Assume $\diamondsuit_{\omega_1}$ and fix the sequences $g^\alpha$
supplied by Lemma~\ref{lem:diamond}. Fix also one map
$\kappa:\N\to\N$ with every fibre infinite, and enumerate each fibre as
\[
 \kappa^{-1}\{k\}=\{m_k(0)<m_k(1)<\cdots\}\qquad(k<\omega).
\]
We use this same map at every successor stage.
We construct a continuous inverse system
\[
 (K_\alpha,\pi_{\alpha,\beta})_{
 \omega\leqslant\alpha\leqslant\beta\leqslant\omega_1},
 \qquad K_\alpha\subseteq2^\alpha,
\]
whose bonding maps are coordinate projections.  Start with $K_\omega=2^\omega$.  Write $\mathcal I_\alpha$ for the set of indices $i$ of measure sequences introduced at or before stage $\alpha$.  Initially $\mathcal I_\omega=\varnothing$: no sequences $(a_{i,n})_n$, index sets $J_i^\omega$ or measures $\mu_i^\omega$ have yet been chosen, so the requirements \eqref{eq:J-invariant}--\eqref{eq:uniform-invariant} are vacuous.  At each limit stage take the inverse limit.  Each successor map described below is onto.  At a limit, compactness and the finite intersection property show that every point of every earlier space extends to a point of the inverse limit.  Thus all bonding maps are surjective.

Inductively every $K_\alpha$ is compact and zero-dimensional: this holds at the start, is preserved by the successor lemma and the constant-coordinate extension, and is preserved by inverse limits. Every proper stage is metrisable, being a subspace of $2^\alpha$ for countable $\alpha$.

Suppose the system and all the measure data have been constructed through a proper stage $\alpha$.  Each earlier successor introduced at most countably many sequences, and $\alpha$ is countable, so $\mathcal I_\alpha$ is countable.  We distinguish the two cases according to the validity of $g^\alpha$, as defined in Subsection~\ref{subsec:prediction}.

If the prediction is invalid, put $K_{\alpha+1}=K_\alpha\times\{0\}\subseteq2^{\alpha+1}$, with constant new coordinate, and let $s_\alpha(x)=(x,0)$.  For each $i\in\mathcal I_\alpha$, set
\[
 J_i^{\alpha+1}=J_i^\alpha,
 \qquad \mu_i^{\alpha+1}=(s_\alpha)_*\mu_i^\alpha.
\]
The map $\pi_{\alpha,\alpha+1}$ is a homeomorphism with inverse $s_\alpha$, so its push-forward identifies the corresponding lift sets and all three preservation requirements hold.  No new sequences are introduced, and $\mathcal I_{\alpha+1}=\mathcal I_\alpha$.

For a valid prediction, apply Lemma~\ref{lem:flat} to $A_\alpha$, obtaining disjointly supported measures
\[
 \sigma_n=p_n-q_n\in A_\alpha,
 \qquad p_n,q_n\xrightarrow{w^*}\lambda_\alpha,
 \qquad\|\sigma_n\|=2.
\]
Enumerate the clopen subsets of $K_\alpha$ without repetitions as $(D_{\alpha,k})_{k<\omega}$.  This is possible because $K_\alpha$ is compact metrisable and zero-dimensional; its clopen algebra is countable and infinite, since $K_\alpha$ maps onto $2^\omega$.  Apply Lemma~\ref{lem:simultaneous-successor} with this enumeration and the fixed map $\kappa$.  For each $i\in\mathcal I_\alpha$, the required data are
\[
 r_i=\pi_{b_i,\alpha},\quad
 (a_{i,n})_{n<\omega},\quad J_i^\alpha,\quad\mu_i^\alpha.
\]
The compact sets used by that lemma are
\[
 E_{i,n}^\alpha=
 \pi_{b_i,\alpha}^{-1}(\supp a_{i,n});
\]
for each fixed $i$ they are pairwise disjoint.  If $\mathcal I_\alpha$ is finite or empty, the lemma is applied to precisely this finite or empty family of data.

For this application enumerate $\mathcal I_\alpha$ by an initial segment of $\N$, as in the hypotheses of Lemma~\ref{lem:simultaneous-successor}, and then use the original indices $i\in\mathcal I_\alpha$ for the resulting data.

Let $F_\alpha,U_\alpha$ be the sets furnished by the lemma.  Identify the space $K'\subseteq K_\alpha\times\{0,1\}$ from that lemma with $K_{\alpha+1}\subseteq2^{\alpha+1}$, using its second coordinate as coordinate $\alpha$, and put
\[
 s_\alpha(x)=(x,\one_{U_\alpha}(x)).
\]
For each $i\in\mathcal I_\alpha$, take $J_i^{\alpha+1}$ to be the infinite subset of $J_i^\alpha$ supplied by the lemma, and set $\mu_i^{\alpha+1}=(s_\alpha)_*\mu_i^\alpha$.  We check the requirements separately.  First,
$J_i^{\alpha+1}\subseteq J_i^\alpha$, which together with the earlier almost inclusions proves \eqref{eq:J-invariant}.  Secondly,
$\pi_{\alpha,\alpha+1}\circ s_\alpha=\mathrm{id}_{K_\alpha}$, so
$(\pi_{\alpha,\alpha+1})_*\mu_i^{\alpha+1}=\mu_i^\alpha$; the earlier coherence identities then give \eqref{eq:mu-invariant} at every earlier stage.  Thirdly, part~(1) of Lemma~\ref{lem:simultaneous-successor} is exactly \eqref{eq:uniform-invariant} at $\alpha+1$.  The norm bound is retained because Borel push-forward does not increase total variation.

The successor lemma supplies distinct indices $t_m$ and their unique coefficient-preserving lifts $\widehat\sigma_{t_m}\in L(K_{\alpha+1})$.  For the new index $i=(\alpha,k)$, put $b_i=\alpha+1$ and introduce the sequence
\begin{equation}\label{eq:birth-sequence}
 a_{(\alpha,k),n}=
 \frac12\widehat\sigma_{t_{m_k(n)}},
 \qquad n<\omega.
\end{equation}
This is a sequence indexed by $\N$, of norm-one measures with pairwise disjoint finite supports.  Distinct values of $k$ use disjoint sets of indices $m$, so no selected measure occurs in two of the newly introduced sequences.  Set $J_i^{b_i}=\N$ and
\begin{equation}\label{eq:birth-limit}
 \mu_{(\alpha,k)}^{\alpha+1}
 =\frac12\int_{F_\alpha\cap D_{\alpha,k}}
  (\delta_{(x,1)}-\delta_{(x,0)})\,d\lambda_\alpha(x).
\end{equation}
Part~(2) of the successor lemma gives convergence of \eqref{eq:birth-sequence} to \eqref{eq:birth-limit}.  Since $B_{i,n}^{b_i}=\{a_{i,n}\}$, this verifies \eqref{eq:uniform-invariant}.  The other two requirements are identities at the initial stage $b_i$.  A zero initial limit is allowed and requires no special case.  Set $\mathcal I_{\alpha+1}=\mathcal I_\alpha\cup(\{\alpha\}\times\N)$.

At a countable limit $\rho$, put $\mathcal I_\rho=\bigcup_{\omega\leqslant\alpha<\rho}\mathcal I_\alpha$ and apply Lemma~\ref{lem:limit-invariant} to each $i\in\mathcal I_\rho$.  This completes the construction through all proper stages.  Put $K=K_{\omega_1}$; it is a compact zero-dimensional space mapping onto $2^\omega$, and hence infinite, with weight at most $\aleph_1$.

For each index $i$ introduced in the construction, the coherent measures determine a final measure $\mu_i\in M(K)$ directly from the exact factorisation in Lemma~\ref{lem:factor}.  If
$f=g\circ\pi_{\alpha,\omega_1}$ with $\alpha\geqslant b_i$, define
\begin{equation}\label{eq:final-record-measure}
 \mu_i(f)=\mu_i^\alpha(g).
\end{equation}
If two such representations are given, pass to a common proper stage and use surjectivity and \eqref{eq:mu-invariant}, exactly as in the well-definedness argument of Lemma~\ref{lem:limit-invariant}.  The values coincide.  Passing to a common stage also proves linearity, and
$|\mu_i(f)|\leqslant\|f\|_\infty$ follows from surjectivity.  Thus \eqref{eq:final-record-measure} defines a bounded functional on all of $C(K)$, and the Riesz representation theorem identifies it with a signed Radon measure of norm at most one.  Its projection to every $K_\alpha$, $\alpha\geqslant b_i$, is $\mu_i^\alpha$.  There is no choice of a set $J_i^{\omega_1}$ in the construction.

The uniform convergence requirement was imposed for all norm-preserving lifts so that it applies to any measures on the final space having the prescribed initial projections.  The next lemma gives the resulting weak-star closure statement.

\begin{lemma}\label{lem:final-closure}
Let $i$ be an index introduced in the construction and let $A\subseteq L(K)$ be a linear subspace.  If $A$ contains norm-one measures $(w_n)_{n<\omega}$ with
\[
 (\pi_{b_i,\omega_1})_*w_n=a_{i,n}\qquad(n<\omega),
\]
then $\mu_i\in\overline A^{\,w^*}$, where the closure is taken in $M(K)$.
\end{lemma}
\begin{proof}
Let $f_1,\ldots,f_r\in C(K)$ and $\eps>0$ specify a basic weak-star neighbourhood of $\mu_i$.  Lemma~\ref{lem:factor} gives one proper stage $\beta\geqslant b_i$ and functions $g_j\in C(K_\beta)$ with
$f_j=g_j\circ\pi_{\beta,\omega_1}$ for $1\leqslant j\leqslant r$.  The measure
$(\pi_{\beta,\omega_1})_*w_n$ belongs to $B_{i,n}^\beta$ for every $n$.  By \eqref{eq:uniform-invariant}, all sufficiently large $n$ in the infinite set $J_i^\beta$ satisfy, simultaneously for $1\leqslant j\leqslant r$,
\[
 |w_n(f_j)-\mu_i(f_j)|
 =|((\pi_{\beta,\omega_1})_*w_n)(g_j)-\mu_i^\beta(g_j)|
 <\eps.
\]
Thus every such neighbourhood meets $A$.
\end{proof}

\section{From localised limits to differences of point masses}\label{sec:coupling}

Fix a valid prediction at a stage $\theta$, and put
$b=\theta+1$, $F=F_\theta$ and $\lambda_F=\lambda_\theta\restr{F}$.
For a clopen $D\subseteq K_\theta$, write
$\mu_D^\beta=\mu_{(\theta,k)}^\beta$ and
$\mu_D=\mu_{(\theta,k)}$, where $D=D_{\theta,k}$.
We shall show that if all the measures $\mu_D$ belong to the weak-star
closure of $A\subseteq L(K)$, then that closure contains a nonzero
difference of point masses whose projection to $K_\theta$ is zero.
Only the Borel sections at proper stages are needed: compactness will
supply the final pair of points.
Write $\lambda=\lambda_\theta$, so $\lambda_F=\lambda\restr{F}$ and
$\lambda(F)>0$.

For $e\in\{0,1\}$ and $b\leqslant\beta<\omega_1$, we construct Borel maps
\[
 t_e^\beta:F\longrightarrow K_\beta
\]
with coherent projections.  At stage $b$ set $t_e^b(x)=(x,e)$; this is continuous.  At a successor set
$t_e^{\beta+1}=s_\beta\circ t_e^\beta$, which is Borel because both maps are Borel.  Since
$\pi_{\beta,\beta+1}\circ s_\beta=\mathrm{id}$, this definition retains coherence with $t_e^\beta$ and hence with all earlier maps.  At a countable limit $\rho$, define $t_e^\rho(x)$ as the unique point of the inverse limit with the prescribed coherent earlier projections.  Its coordinate at each $\xi<\rho$ is the corresponding coordinate of some $t_e^\beta$ with $b\leqslant\beta<\rho$ and $\xi<\beta$, and is therefore Borel.  Since $\rho$ is countable, $2^\rho$ has a countable base of finite-coordinate cylinders.  Their inverse images are Borel, so $t_e^\rho$ is Borel as a map into $K_\rho$.  Coherence holds by the inverse-limit definition.  The induction also gives
\begin{equation}\label{eq:section-coherence}
 \pi_{\beta,\delta}\circ t_e^\delta=t_e^\beta
 \quad(b\leqslant\beta\leqslant\delta<\omega_1),
 \qquad \pi_{\theta,\beta}(t_e^\beta(x))=x.
\end{equation}

At every proper stage the localisation formula is the following identity of Radon measures:
\begin{equation}\label{eq:localisation-proper}
 \mu_D^\beta=
 \frac12\bigl((t_1^\beta)_*-(t_0^\beta)_*\bigr)
 (\lambda\restr{F\cap D}).
\end{equation}
All push-forwards on the right are finite Borel measures on the compact metrisable space $K_\beta$ and are therefore Radon.  At stage $b$ the identity is \eqref{eq:birth-limit}.  If it holds at $\beta$, push both sides forward by the same Borel map $s_\beta$.  The composition rule for Borel push-forwards and the definition of $t_e^{\beta+1}$ give the identity at $\beta+1$.  This is legitimate for bounded Borel functions: equality of Radon measures entails equality of their integrals against such functions, even though $f\circ s_\beta$ need not be continuous.

At a countable limit $\rho$, let the right-hand side of \eqref{eq:localisation-proper} define $\nu_D^\rho$.  For every $\beta\in[b,\rho)$, coherence \eqref{eq:section-coherence} gives
$(\pi_{\beta,\rho})_*\nu_D^\rho=\mu_D^\beta$.
The same holds for $\mu_D^\rho$ by \eqref{eq:mu-invariant}.  Hence these two measures agree on every continuous function pulled back from a stage below $\rho$.  These functions are uniformly dense by the Stone--Weierstrass argument in Lemma~\ref{lem:limit-invariant}; both measures are bounded, so they agree on all of $C(K_\rho)$.  This proves \eqref{eq:localisation-proper} at every proper stage.

The fixed relation
\begin{equation}\label{eq:initial-relation}
 R=\{(y_0,y_1)\in K^2:
       \pi_{\theta,\omega_1}(y_0)=\pi_{\theta,\omega_1}(y_1)\in F,
       \ (y_0)_\theta=0,\ (y_1)_\theta=1\}
\end{equation}
is compact and disjoint from the diagonal. It is nonempty: for any
$x\in F$, surjectivity lets us lift $(x,0),(x,1)\in K_b$ to $K$.

\begin{lemma}\label{lem:coupling}
Let $A\subseteq L(K)$ be a linear subspace such that
$\mu_D\in\overline A^{\,w^*}$ for every clopen $D\subseteq K_\theta$,
where the closure is in $M(K)$. There is a pair $(y_0,y_1)\in R$ such that
\[
 f(y_0)=f(y_1)\quad(f\in A_\perp).
\]
Consequently
$0\neq\delta_{y_1}-\delta_{y_0}\in\overline A^{\,w^*}\cap L(K)$,
whereas $(\pi_{\theta,\omega_1})_*(\delta_{y_1}-\delta_{y_0})=0$.
\end{lemma}
\begin{proof}
For $f\in A_\perp$, let
\[
 R_f=\{(y_0,y_1)\in R:f(y_0)=f(y_1)\}.
\]
These sets are closed in the compact space $R$. We prove that they
have the finite intersection property.

Fix $f_1,\ldots,f_r\in A_\perp$. Lemma~\ref{lem:factor} gives one
proper stage $\beta\geqslant b$ and functions
$\bar f_j\in C(K_\beta)$ such that
$f_j=\bar f_j\circ\pi_{\beta,\omega_1}$ for every $j$.
The bounded Borel functions
\[
 h_j=\bar f_j\circ t_1^\beta-\bar f_j\circ t_0^\beta
 \quad\text{on }F
\]
satisfy
\[
 \int_{F\cap D}h_j\,d\lambda
 =2\mu_D^\beta(\bar f_j)
 =2\mu_D(f_j)=0
 \qquad(D\subseteq K_\theta\text{ clopen}).
\]
Here the first two equalities follow from
\eqref{eq:localisation-proper} and \eqref{eq:final-record-measure};
the last follows because every member of $\overline A^{\,w^*}$
annihilates $f_j$.
The traces on $F$ of clopen subsets of $K_\theta$ form an algebra
generating the Borel sets of $F$. Uniqueness of finite signed measures
on a generating algebra therefore gives $h_j\lambda_F=0$, and hence
$h_j=0$ almost everywhere with respect to $\lambda_F$.
Since $\lambda_F(F)>0$, choose $x\in F$ at which all of
$h_1,\ldots,h_r$ vanish. Lift the two points
$t_0^\beta(x),t_1^\beta(x)$ separately to points $y_0,y_1\in K$.
Coherence \eqref{eq:section-coherence} puts $(y_0,y_1)$ in $R$, and
factorisation gives
\[
 f_j(y_1)-f_j(y_0)=h_j(x)=0\qquad(1\leqslant j\leqslant r).
\]
Thus $\bigcap_{j=1}^rR_{f_j}\neq\varnothing$.

Compactness now gives a pair in $\bigcap_{f\in A_\perp}R_f$.
Its difference of point masses annihilates $A_\perp$, so the bipolar
identity \eqref{eq:bipolar} places it in
$\overline A^{\,w^*}\cap L(K)$. Relation~\eqref{eq:initial-relation}
shows that the difference is nonzero and has zero projection to
$K_\theta$.
\end{proof}

\section{Proof of Theorem A}\label{sec:finish}

We first prove the quotient assertions of Theorem~A.

\begin{proof}[Proof of the quotient assertions]
Let $K$ be the compactum constructed above.  By Theorem~B and Lemma~\ref{lem:dual}, the quotient assertions follow once every countably infinite-dimensional subspace of $L(K)$ is shown not to be relatively weak-star closed.  Each such subspace will be treated at one correctly predicted stage of the construction.

Fix
\[
 A=\Span\{v_n:n<\omega\}\subseteq L(K),
 \qquad\dim A=\aleph_0,
 \qquad S=\bigcup_{n<\omega}\supp v_n.
\]
Here $L(S)$ denotes the measures in $L(K)$ whose finite supports are contained in $S$; the countable set $S$ need not be closed.  In particular, $A\subseteq L(S)$.  Let us regard the $v_n$ as measures on $2^{\omega_1}$. Lemma~\ref{lem:diamond} gives a correct prediction stage $\theta$ at which $\pi_{\theta,\omega_1}$ is injective on $S$.

Distinct atoms of any measure in $L(S)$ remain distinct under this projection.  Thus push-forward is injective and variation-isometric on $L(S)$.  In particular, the predicted span at stage $\theta$ is
\[
 A_\theta=
 \Span\{(\pi_{\theta,\omega_1})_*v_n:n<\omega\}
 =(\pi_{\theta,\omega_1})_*A,
\]
and it is infinite dimensional.  Its measures are supported in $K_\theta$, so this is a valid prediction.  The restriction of push-forward to $A$ is consequently a linear isometric bijection from $A$ onto $A_\theta$.

Let $(\sigma_n)$ be the sequence chosen from $A_\theta$ at this successor step.  There are unique $w_n\in A$ with
\[
 (\pi_{\theta,\omega_1})_*w_n=\sigma_n,
 \qquad \|w_n\|=\|\sigma_n\|=2.
\]
Their supports are pairwise disjoint: an atom belonging to two of them would project to an atom in the two corresponding disjoint supports of the $\sigma_n$.  Put $b=\theta+1$.  For each selected index $t_m$, define
\[
 \tau_m=(\pi_{b,\omega_1})_*w_{t_m}.
\]
Composition of the projections and their norm contractivity give
\begin{equation}\label{eq:norm-preserving-hinge}
 2=\|\sigma_{t_m}\|
 =\|(\pi_{\theta,b})_*\tau_m\|
 \leqslant\|\tau_m\|
 \leqslant\|w_{t_m}\|=2.
\end{equation}
Hence $\tau_m/2$ is a norm-preserving lift of $\sigma_{t_m}/2$.  Lemma~\ref{lem:lifts} shows that it is concentrated on
$\pi_{\theta,b}^{-1}(\supp\sigma_{t_m})$.
The successor construction makes every fibre over $\supp\sigma_{t_m}$ a singleton, because that support misses $F_\theta$.  Thus the projection on this finite preimage is a bijection, and its prescribed push-forward fixes the coefficient at every point.  It follows that
\begin{equation}\label{eq:identified-lift}
 (\pi_{b,\omega_1})_*w_{t_m}
 =\tau_m=\widehat\sigma_{t_m}.
\end{equation}

Fix any $k<\omega$.  Reindexing the norm-one final measures
\[
 \left(\frac12w_{t_{m_k(n)}}\right)_{n<\omega}
 \subseteq A
\]
according to \eqref{eq:birth-sequence}, identity \eqref{eq:identified-lift} shows that their projections to $K_b$ are exactly the sequence $(a_{(\theta,k),n})_{n<\omega}$.  Lemma~\ref{lem:final-closure} now gives
\[
 \mu_D\in\overline A^{\,w^*}
 \qquad(D\subseteq K_\theta\text{ clopen}).
\]

Lemma~\ref{lem:coupling} gives a pair $(y_0,y_1)\in R$ such that
\[
0\neq\eta=\delta_{y_1}-\delta_{y_0}
\in\overline A^{\,w^*}\cap L(K),
\qquad (\pi_{\theta,\omega_1})_*\eta=0.
\]
Push-forward is injective on $A\subseteq L(S)$, so $\eta\notin A$. Hence $A$ is not relatively weak-star closed.

Lemma~\ref{lem:dual} excludes infinite-dimensional metrisable quotients, and Theorem~B excludes infinite-dimensional separable quotients.
\end{proof}

The remaining assertions of Theorem~A follow from the topological properties proved in the next section.

\section{Further properties of the compact space}\label{sec:properties}

We first verify that the successor maps used above preserve enough of the topology of the initial Cantor space to give separability and crowdedness at the final stage.

\begin{lemma}\label{lem:irreducible-successor}
Every projection produced by Lemma~\ref{lem:simultaneous-successor} is irreducible: no proper closed subset of its domain maps onto its range.
\end{lemma}
\begin{proof}
Retain the notation of that lemma. We already know $F\subseteq\overline U$. We claim also that
\begin{equation}\label{eq:lower-boundary}
F\subseteq\overline{K\setminus(U\cup F)}.
\end{equation}
Fix $x\in F$ and a clopen neighbourhood $V$ of $x$. Then $\lambda(F\cap V)>0$ by \eqref{eq:F-properties}. Consider sufficiently large steps $m$ for which $D_{\kappa(m)}=K$ and $\one_V$ occurs among $f_0,\ldots,f_m$. Since $C_m\cap P_m=\varnothing$, equations \eqref{eq:fresh-rank1}--\eqref{eq:fresh-rank3} give
\begin{align*}
q_{t_m}(V\setminus P_m)
&=p_{t_m}(V\setminus P_m)-\sigma_{t_m}(V)
                                  +\sigma_{t_m}(V\cap P_m)\\
&\geqslant p_{t_m}(C_m\cap V)-2\cdot2^{-m}\\
&\geqslant\lambda(F\cap V)-3\cdot2^{-m}>0.
\end{align*}
Choose a point of $\supp q_{t_m}$ in $V\setminus P_m$. It does not belong to $T_m\subseteq\supp p_{t_m}$, so it belongs to $L_m$. By construction, $L_m\cap F=\varnothing$ and $L_m\cap U_j=\varnothing$ for every $j<\omega$. Our point therefore lies in $V\setminus(U\cup F)$, proving \eqref{eq:lower-boundary}.

The singleton fibres of $\pi$ are those over $K\setminus F$. They are dense in $K'$: the points of $U\times\{1\}$ are dense at every $(x,1)$ with $x\in F$ because $F\subseteq\overline U$, and the points of $(K\setminus(U\cup F))\times\{0\}$ are dense at every $(x,0)$ with $x\in F$ by \eqref{eq:lower-boundary}. A closed subset of $K'$ mapping onto $K$ must contain every singleton fibre and hence their closure $K'$. This proves irreducibility.
\end{proof}

Irreducibility passes to the limit and transfers separability and crowdedness from the initial Cantor space.

\begin{proposition}\label{prop:separable-crowded}
The space $K$ constructed above is separable and crowded.
\end{proposition}
\begin{proof}
For a continuous surjection $r:X\to Y$ between compact Hausdorff spaces, irreducibility is equivalent to the following condition: every nonempty open $V\subseteq X$ contains $r^{-1}(W)$ for some nonempty open $W\subseteq Y$. Indeed, take $W=Y\setminus r[X\setminus V]$; it is open by compactness and nonempty precisely when the proper closed set $X\setminus V$ fails to map onto $Y$. This condition also shows that a composition of irreducible maps is irreducible.

We prove by transfinite induction that $\pi_{\omega,\rho}:K_\rho\to2^\omega$ is irreducible for every $\omega\leqslant\rho\leqslant\omega_1$. The assertion is immediate at $\rho=\omega$. At a successor, the new bonding map is irreducible by Lemma~\ref{lem:irreducible-successor}, or is a homeomorphism in the invalid case, and composition applies. At a limit $\rho$, a nonempty open subset of $K_\rho$ contains a nonempty cylinder $\pi_{\beta,\rho}^{-1}(V)$ for some $\omega\leqslant\beta<\rho$ and nonempty open $V\subseteq K_\beta$. The inductive hypothesis supplies a nonempty open $W\subseteq2^\omega$ with $\pi_{\omega,\beta}^{-1}(W)\subseteq V$. Consequently $\pi_{\omega,\rho}^{-1}(W)$ lies in the original open set, as required. This argument includes the final limit $\rho=\omega_1$.

Put $r=\pi_{\omega,\omega_1}$. Choose a countable dense subset $D\subseteq2^\omega$ and one point $x_d\in r^{-1}(d)$ for each $d\in D$. The open-set criterion just proved shows that every nonempty open subset of $K$ contains some $x_d$, so these points form a countable dense set. If $x\in K$ were isolated, that criterion applied to $\{x\}$ would give a nonempty open $W\subseteq2^\omega$ with $r^{-1}(W)\subseteq\{x\}$. Surjectivity would force $W=\{r(x)\}$, contradicting the absence of isolated points in the Cantor space.
\end{proof}

The quotient conclusion also excludes the two subspaces appearing in the definition of an Efimov space.

\begin{corollary}\label{cor:efimov}
The space $K$ is an Efimov space and has weight exactly $\aleph_1$.
\end{corollary}
\begin{proof}
An infinite compact space containing a nontrivial convergent sequence or a copy of $\beta\N$ has an infinite-dimensional metrisable quotient of its pointwise function space, by \cite[Corollary~3]{BKS2018}. The conclusion proved in Section~\ref{sec:finish} therefore excludes both subspaces of $K$. An infinite compact metrisable space contains a nontrivial convergent sequence, so $K$ is not metrisable. Since $K\subseteq2^{\omega_1}$, its weight is at most $\aleph_1$; a compact Hausdorff space of countable weight is metrisable. Thus its weight is $\aleph_1$.
\end{proof}

The absence of convergent sequences determines the cardinality through a binary tree of closed $G_\delta$ sets.

\begin{proposition}\label{prop:cardinality}
The cardinality of $K$ is $2^{\aleph_1}$.
\end{proposition}
\begin{proof}
There are no $G_\delta$ points in $K$. To see this, suppose $\{x\}=\bigcap_n O_n$ with each $O_n$ open. By regularity choose open neighbourhoods $V_n$ of $x$ such that $\overline V_0\subseteq O_0$ and $\overline V_{n+1}\subseteq V_n\cap O_{n+1}$. Compactness shows that $(V_n)$ is a local base at $x$: if $O$ is a neighbourhood of $x$ and no $\overline V_n$ is contained in $O$, the nested nonempty compact sets $\overline V_n\setminus O$ have an intersection, contradicting $\bigcap_n\overline V_n=\{x\}$. In a crowded Hausdorff space every nonempty open set is infinite. We can therefore choose pairwise distinct $x_n\in V_n\setminus\{x\}$, giving a nontrivial convergent sequence, contrary to Corollary~\ref{cor:efimov}.

It follows that every nonempty closed $G_\delta$ subset of $K$ has at least two points. Construct nonempty closed $G_\delta$ sets $F_s$, for $s\in2^{<\omega_1}$, starting with $F_\varnothing=K$. This is a set-sized recursion: $|2^{<\omega_1}|=2^{\aleph_0}$, since each countable level has cardinality at most $2^{\aleph_0}$, there are $\aleph_1\leqslant2^{\aleph_0}$ levels, and the level $2^\omega$ has cardinality $2^{\aleph_0}$. Given $F_s$, choose distinct $x_0,x_1\in F_s$ and a continuous $u:K\to[0,1]$ with $u(x_e)=e$. Set
\[
F_{s^\frown0}=F_s\cap u^{-1}([0,1/3]),\qquad
F_{s^\frown1}=F_s\cap u^{-1}([2/3,1]).
\]
These are disjoint nonempty closed $G_\delta$ sets: inverse images of the displayed closed intervals are zero sets and hence $G_\delta$. At a nonzero countable limit $\rho$, put $F_s=\bigcap_{\xi<\rho}F_{s\restriction\xi}$ for $s\in2^\rho$. This is a nonempty closed $G_\delta$ set by compactness and countability, so the recursion continues. For each branch $z\in2^{\omega_1}$, compactness gives a point in $\bigcap_{\xi<\omega_1}F_{z\restriction\xi}$. Distinct branches give disjoint intersections, yielding $2^{\aleph_1}$ distinct points of $K$. The reverse inequality follows from $K\subseteq2^{\omega_1}$.
\end{proof}

The classical theorem of Josefson and Nissenzweig \cite{Josefson,Nissenzweig} supplies norm-one weak-star null sequences in the dual of every infinite-dimensional Banach space; finite support is an additional requirement. Write
\[
\ell_1(X)=\left\{\sum_{x\in S}a_x\delta_x:
 S\subseteq X\text{ countable},\ \sum_{x\in S}|a_x|<\infty\right\}
\subseteq M(X).
\]
The space $C(X)$ has the \emph{$\ell_1$-Grothendieck property} if every weak-star convergent sequence in $\ell_1(X)$ is weakly convergent in $M(X)$.

\begin{corollary}\label{cor:fsjn}
The space $K$ admits no finitely supported Josefson--Nissenzweig sequence, and $C(K)$ has the $\ell_1$-Grothendieck property.
\end{corollary}
\begin{proof}
By \cite[Theorem~1]{BKS2019}, the existence of such a sequence is equivalent to $C_p(K)$ having a quotient isomorphic to $(c_0)_p$, where $(c_0)_p$ is $c_0$ with the topology inherited from $\R^\N$. This would be an infinite-dimensional Hausdorff metrisable quotient. The second assertion follows from \cite[Theorem~4.1]{KSZGrothendieck}.
\end{proof}

The full Grothendieck property of $C(K)$ and the possible existence of an isomorphic copy of $\ell_\infty$ in $C(K)$ are not decided by this argument.

\section{An obstruction to one-point resolutions}\label{sec:obstruction}

The following elementary example explains why a single double fibre need not disturb the convergence of a sequence of differences of probability measures.

\begin{proposition}\label{prop:one-point}
On the Cantor space $X=2^\omega$ there are finitely supported probabilities $p_n,q_n$ such that all their supports are pairwise disjoint and $p_n,q_n\to\lambda$ weak-star for the nonatomic Bernoulli probability $\lambda$. If $\pi:Y\to X$ is a continuous surjection from a compact metrisable space and all its fibres are singletons except possibly one double fibre, then every choice of lifts $\widehat p_n,\widehat q_n$ obtained by placing each atom over its original point and retaining its coefficient satisfies
\[
\frac12(\widehat p_n-\widehat q_n)\xrightarrow{w^*}0,
\qquad
\left\|\frac12(\widehat p_n-\widehat q_n)\right\|=1.
\]
In particular, $Y$ admits a finitely supported Josefson--Nissenzweig sequence.
\end{proposition}
\begin{proof}
For each $n$ and each binary word $s$ of length $n$, choose two points $x^0_{n,s},x^1_{n,s}$ in its basic cylinder $[s]$, choosing every point distinct from all points chosen earlier. This is possible since every cylinder is infinite and at each step only finitely many points have been used. Set
\[
p_n=2^{-n}\sum_{s\in2^n}\delta_{x^0_{n,s}},\qquad
q_n=2^{-n}\sum_{s\in2^n}\delta_{x^1_{n,s}}.
\]
For $f\in C(X)$, both $|p_n(f)-\lambda(f)|$ and $|q_n(f)-\lambda(f)|$ are bounded by $\max_{s\in2^n}\operatorname{osc}(f,[s])$, which tends to zero by uniform continuity. This proves the asserted convergence.

Let $x$ be the exceptional base point, choosing any $x\in X$ if there is no exceptional fibre. We show that $\lambda$ has a unique probability lift to $Y$. On $X\setminus\{x\}$ the inverse of the singleton-fibre map is continuous. Indeed, if $y\neq x$ and an open set $O\subseteq Y$ contains the unique point over $y$, then $X\setminus\pi[Y\setminus O]$ is an open neighbourhood of $y$ whose full preimage lies in $O$. Thus for every $h\in C(Y)$, the function $g(y)=h(\pi^{-1}(y))$ on $X\setminus\{x\}$ extends, by assigning any value at $x$, to a bounded Borel function on $X$. Any probability $\nu$ with $\pi_*\nu=\lambda$ gives mass zero to $\pi^{-1}(x)$, since $\lambda(\{x\})=0$, and therefore
\[
\nu(h)=\int_X g\,d\lambda.
\]
This identity determines $\nu$ uniquely on $C(Y)$. Existence follows, for example, by choosing finite probability lifts of the $p_n$ and taking a weak-star cluster point in the compact space $P(Y)$.

For any such choices, every cluster point of $(\widehat p_n)$ and of $(\widehat q_n)$ is a probability projecting to $\lambda$. Both sequences therefore have the same unique cluster point. Since $P(Y)$ is compact metrisable, each sequence converges to that point: otherwise a subsequence outside a fixed neighbourhood would have a different cluster point. Their difference tends to zero. Disjointness of the original supports implies disjointness of all lifted supports, so each difference has norm two, as required.
\end{proof}

The same probability measures show directly that stationarily many successor maps in our construction double a perfect uncountable set.

\begin{proposition}\label{prop:perfect-splitting-stages}
In the inverse system of Section~\ref{sec:recursion}, stationarily many valid stages $\theta<\omega_1$ have nonatomic $\lambda_\theta$ and perfect, uncountable $F_\theta$.
\end{proposition}
\begin{proof}
Let $\rho$ be the Bernoulli probability on $K_\omega=2^\omega$ and choose finitely supported probabilities $p_j^0,q_j^0\to\rho$ with all supports pairwise disjoint, as in Proposition~\ref{prop:one-point}. Lift each atom once to $K=K_{\omega_1}$ and retain its coefficient, obtaining probabilities $p_j,q_j\in L(K)$. Put $v_j=p_j-q_j$ and $A=\Span\{v_j:j<\omega\}$. The initial projection is injective on the union of all these supports. Consequently, for every $\omega\leqslant\theta<\omega_1$, the measures
\[
p_j^\theta=(\pi_{\theta,\omega_1})_*p_j,\qquad
q_j^\theta=(\pi_{\theta,\omega_1})_*q_j
\]
have pairwise disjoint supports, and $v_j^\theta=p_j^\theta-q_j^\theta$ has norm two.

Lemma~\ref{lem:diamond} supplies stationarily many stages at which the predicted subspace is
$A_\theta=\Span\{v_j^\theta:j<\omega\}$. The vectors $v_j^\theta$ are linearly independent by disjointness of their supports, so every such prediction is valid. Fix one of these stages and write the sequence selected by Lemma~\ref{lem:flat} as
\[
\sigma_n=\sum_jc_{n,j}v_j^\theta,\qquad
I_n=\{j:c_{n,j}\neq0\}.
\]
The sums are finite, and disjointness of all block supports gives
$2=\|\sigma_n\|=2\sum_j|c_{n,j}|$.
Furthermore, the finite sets $I_n$ are pairwise disjoint. Indeed, if $c_{n,j}\neq0$, the entire support of $v_j^\theta$ is contained in $\supp\sigma_n$, because no other block can cancel any of its coefficients. The supports of the $\sigma_n$ are pairwise disjoint. Thus $I_n$ eventually avoids every finite subset of $\N$.

For $c\in\R$, set $c^+=\max\{c,0\}$ and $c^-=\max\{-c,0\}$. The Jordan parts satisfy
\[
\begin{split}
(\pi_{\omega,\theta})_*\sigma_n^+
 &=\sum_j(c_{n,j}^+p_j^0+c_{n,j}^-q_j^0),\\
(\pi_{\omega,\theta})_*\sigma_n^-
 &=\sum_j(c_{n,j}^+q_j^0+c_{n,j}^-p_j^0).
\end{split}
\]
These are convex combinations with total coefficient one. For $f\in C(2^\omega)$ and $N<\omega$, put
\[
\eps_N(f)=\sup_{j\geqslant N}
\max\{|p_j^0(f)-\rho(f)|,\ |q_j^0(f)-\rho(f)|\}.
\]
We have $\eps_N(f)\to0$. When $I_n\subseteq[N,\infty)$, the absolute difference between the evaluation of either displayed probability at $f$ and $\rho(f)$ is at most
$\sum_j|c_{n,j}|\eps_N(f)=\eps_N(f)$.
Hence both projected Jordan sequences converge weak-star to $\rho$. Their limits at stage $\theta$ are $\lambda_\theta$, so
$(\pi_{\omega,\theta})_*\lambda_\theta=\rho$.

The measure $\lambda_\theta$ is nonatomic. Indeed, the inverse images of the length-$r$ Cantor cylinders partition $K_\theta$ into finitely many sets of measure $2^{-r}$. For any Borel set of positive measure, choose $r$ with $2^{-r}$ smaller than that measure. One member of this partition cuts off a strictly smaller positive measure, ruling out an atom. The nonzero restriction of $\lambda_\theta$ used in Lemma~\ref{lem:simultaneous-successor} is therefore nonatomic, and that lemma gives a perfect, uncountable $F_\theta$. This holds at every correct prediction stage chosen above.
\end{proof}

The theorem on simple extensions recalled in Section~\ref{sec:prelim}, together with Corollary~\ref{cor:fsjn}, precludes a presentation of $K$ by a continuous inverse system starting with a singleton and using simple extensions at every successor. Proposition~\ref{prop:perfect-splitting-stages} identifies the difference directly in our system: stationarily many successor maps have uncountably many two-point fibres. At each such stage the map is also not fully closed, since the images of the disjoint closed sets $K_{\theta+1}\cap(K_\theta\times\{0\})$ and $K_{\theta+1}\cap(K_\theta\times\{1\})$ intersect in the infinite set $F_\theta$, whereas full closedness requires this intersection to be finite.

Fedorchuk's classical diamond example \cite{Fedorchuk} has the same topological properties listed in Theorem~A, but has the finitely supported Josefson--Nissenzweig property \cite[Corollary~6.12]{KSZGrothendieck}. Our example therefore differs in its behaviour for finitely supported measures. Talagrand's compactum \cite{Talagrand}, whose $C(K)$ has the Grothendieck property, also admits no presentation by simple extensions, by \cite[Theorems~4.1 and 6.10]{KSZGrothendieck}.

\section{Relative consistency and the continuum}\label{sec:consistency}

\subsection{Independence at weight $\aleph_1$}

We use $\Con(T)$ for the assertion that the first-order theory $T$ has no formal proof of a contradiction. Statements of relative consistency concern implications between such assertions; they do not assert that the extra hypotheses hold in the universe under consideration.

\begin{proof}[Proof of Theorem~D]
The standard consistency theorem for Martin's axiom gives
\[
\Con(\mathrm{ZFC})\ \Longrightarrow\
\Con(\mathrm{ZFC}+\mathrm{MA}+\cont=\aleph_2);
\]
see \cite{ST,Jech}. Martin's axiom implies $\covN=\cont$ \cite{Jech}. In that model, Theorem~C applies to every infinite compact space of weight $\aleph_1$, giving even a split quotient onto $(c_0)_p$.

For the negative direction, Jensen's theorem gives $\diamondsuit_{\omega_1}$ in the constructible universe \cite{Jensen}. Thus
\[
\Con(\mathrm{ZFC})\ \Longrightarrow\
\Con(\mathrm{ZFC}+\diamondsuit_{\omega_1}).
\]
By Theorem~A, diamond yields a compact space of weight $\aleph_1$ whose pointwise function space has no infinite-dimensional Hausdorff separable quotient. This proves consistency of the negation. If ZFC proved either sentence, it would contradict consistency of the theory obtained by adjoining the other one.
\end{proof}

\subsection{A real-valued measurable continuum}

Put $\kappa=\cont$. The assertion $\RVM(\cont)$ means that there is a probability measure $\nu$ on $\mathcal P(\kappa)$, zero on singletons, such that
\begin{equation}\label{eq:additivity}
\nu\left(\bigcup_{\xi<\lambda}A_\xi\right)
 =\sum_{\xi<\lambda}\nu(A_\xi)
 \qquad(\lambda<\kappa)
\end{equation}
whenever the $A_\xi$ are pairwise disjoint; the sum is the supremum of its finite subsums. In particular, countable additivity holds and unions of fewer than $\kappa$ null sets are null. This is stronger than merely asking for some countably additive probability on the full power set of the reals. It also specifies the continuum itself, rather than an unspecified real-valued measurable cardinal.

\begin{theorem}\label{main:rvm}
Assume $\RVM(\cont)$, and let $K$ be an infinite compact Hausdorff space.
\begin{enumerate}[label=\textup{(\roman*)}]
\item If $w(K)<\cont$, then $K$ satisfies the conclusion of Theorem~C.
\item If $C_p(K)$ has no infinite-dimensional Hausdorff separable quotient, then
\[
w(\overline S)=\cont
\qquad\text{for every countably infinite }S\subseteq K.
\]
Moreover, each such $\overline S$ is itself a compact counterexample. Hence the unrestricted positive assertion reduces to separable compacta of weight exactly $\cont$.
\end{enumerate}
\end{theorem}

\begin{lemma}\label{lem:rvm}
The assertion $\RVM(\cont)$ implies $\covN=\cont$.
\end{lemma}
\begin{proof}
First, $\nu$ is atomless. Fix an injection $b:\kappa\to2^{\N}$. If $A$ were a positive atom, split the current set according to each successive coordinate of $b$ and retain the piece having the full measure of $A$. This gives a decreasing sequence of measurable subsets, all of measure $\nu(A)$, whose intersection contains at most one point. Countable additivity would give that intersection measure $\nu(A)>0$, contrary to the singleton assumption.

An atomless finite measure can divide any measurable set into two subsets of equal measure. Recursively construct partitions $(A_s)_{s\in2^n}$ of $\kappa$ with $\nu(A_s)=2^{-n}$ and $A_{s0}\cup A_{s1}=A_s$. The resulting map $b_0:\kappa\to2^{\N}$ sends a point to its successive partition labels. Its push-forward agrees with the fair product probability on Borel sets, since it has the prescribed values on all finite cylinders.

If fewer than $\kappa$ Borel null sets covered $2^{\N}$, their inverse images under $b_0$ would be fewer than $\kappa$ null sets covering $\kappa$. Disjointifying this cover contradicts~\eqref{eq:additivity}. Every null set is contained in a Borel null set. Hence $\covN\geqslant\kappa$, and the reverse inequality follows by covering the Cantor space by its singletons.
\end{proof}

Theorem~\ref{main:rvm}(i) follows immediately from this lemma and Theorem~C.

We shall use the following standard fact about restriction operators; see Arhangel'skii \cite[Proposition~0.4.1(2)]{Arhangelskii}. We retain a short proof for completeness.

\begin{lemma}\label{lem:restriction}
If $F$ is closed in compact $K$, restriction $r_F:C_p(K)\to C_p(F)$ is an open surjection. Thus a separable infinite-dimensional quotient of $C_p(F)$ induces one of $C_p(K)$.
\end{lemma}
\begin{proof}
Tietze extension gives surjectivity. Let
\[
U=\{f\in C(K):|f(x)|<\eps_x\ (x\in E)\},
\]
where $E\subseteq K$ is finite and $\eps_x>0$. We claim that $r_F(U)$ consists exactly of the $h\in C(F)$ satisfying $|h(x)|<\eps_x$ for $x\in E\cap F$. Extend such an $h$ to $H\in C(K)$. For each $x\in E\setminus F$, choose $b_x\in C(K)$ equal to one at $x$ and zero on $F\cup(E\setminus\{x\})$. Replacing $H$ by $H-\sum_{x\in E\setminus F}H(x)b_x$ retains its restriction to $F$ and sets its values on $E\setminus F$ to zero. This proves the claim and openness. Compose open surjections for the last assertion.
\end{proof}

\begin{proof}[Proof of Theorem~\ref{main:rvm}(ii)]
Let $S\subseteq K$ be countably infinite and put $F=\overline S$. Lemma~\ref{lem:restriction} shows that $C_p(F)$ also has no infinite-dimensional Hausdorff separable quotient. Part~(i) therefore gives $w(F)\geqslant\cont$. On the other hand, $F$ is separable compact. Restriction to $S$ injects $C(F)$ into $\R^S$, so $|C(F)|\leqslant\cont$. Continuous real functions generate the topology of $F$, whence $w(F)\leqslant\cont$. Equality follows.
\end{proof}

In particular, under $\RVM(\cont)$ the universal quotient assertion for infinite compacta is equivalent to its restriction to infinite separable compacta of weight exactly $\cont$. For the converse direction, take the closure of a countably infinite subset of an arbitrary infinite compactum. Its weight is at most $\cont$; apply part~(i) below $\cont$, the restricted assertion at $\cont$, and then Lemma~\ref{lem:restriction}.

An \emph{ordinary measurable cardinal} is an uncountable cardinal $\kappa$ carrying a nonprincipal $\kappa$-complete ultrafilter, equivalently a singleton-zero $\{0,1\}$-valued probability on $\mathcal P(\kappa)$ with the additivity convention in~\eqref{eq:additivity}. Solovay's theorem \cite{Solovay} supplies the different relative-consistency implication
\begin{equation}\label{eq:solovay}
\Con(\mathrm{ZFC}+\text{there is a measurable cardinal})
\ \Longrightarrow\
\Con(\mathrm{ZFC}+\RVM(\cont)).
\end{equation}
In the usual construction, a measurable cardinal in a ground model is used to obtain a real-valued measurable continuum in a forcing extension. Combining~\eqref{eq:solovay} with Theorem~\ref{main:rvm}(i) gives a model of the positive small-weight assertion together with that measure hypothesis. It does not prove the unrestricted $C_p$ assertion.

The measurable-cardinal assumption in~\eqref{eq:solovay} must not be suppressed. ZFC does not provide measurable cardinals. Indeed, a measurable cardinal is strongly inaccessible and yields a set model $V_\kappa\models\mathrm{ZFC}$; the usual incompleteness argument consequently prevents consistent ZFC from proving that such a cardinal exists. A theorem of the form $\mathrm{ZFC}\vdash(\RVM(\cont)\Rightarrow\Psi)$ is valid without establishing that its antecedent has a model. To obtain the latter by~\eqref{eq:solovay}, one assumes consistency of the stronger large-cardinal theory, not merely consistency of ZFC.

Ordinary measurability and real-valued measurability of the continuum should also not be conflated. An ordinary measurable cardinal is a strong limit, whereas $\cont=2^{\aleph_0}$ cannot be a strong-limit cardinal. Thus $\RVM(\cont)$ does not say that the continuum is an ordinary measurable cardinal. Nor does the relative-consistency construction assert that an ordinary measurable cardinal exists in the resulting universe. The relevant distinctions are standard in the theory of large cardinals; see \cite{Solovay,Jech}.

By contrast, no such extra assumption enters Theorem~D. Its two models are obtained from the sole consistency assumption on ZFC, through Martin's axiom and diamond respectively. The small-weight analytic assertion itself is therefore consistent relative to ZFC alone; the measurable cardinal is needed only for the particular route that also supplies $\RVM(\cont)$.

\subsection{Large cardinals alone}

The usual large-cardinal existence assumptions considered below do not force a positive answer. More precisely, the existence of a proper class of measurable cardinals, or even of supercompact cardinals, is relatively consistent with $\neg\SQp(\aleph_1)$, provided the corresponding large-cardinal theory is consistent.

Indeed, force with the end-extension order of countable sequences $(d_\xi)_{\xi<\alpha}$, where $\alpha<\omega_1$ and $d_\xi\subseteq\xi$. This forcing is countably closed, has ground-model size $\cont$, and forces $\diamondsuit_{\omega_1}$; see \cite[Lemma~0.2]{Jackson}. It therefore preserves $\omega_1$. Every measurable or supercompact cardinal exceeds the ground-model continuum, and small forcing preserves its respective large-cardinal property \cite[Theorems~9.6 and~9.8]{Cummings}. A proper class of such cardinals consequently remains a proper class in the extension, where Theorem~A gives the counterexample.

Under the convention that real-valued measurability requires singleton-zero $<\kappa$-additive measures, ordinary measurable cardinals are special cases. Thus even a proper class of real-valued measurable cardinals is compatible with the negative conclusion. If one additionally requires each witnessing measure to be atomless, there cannot be a proper class of such cardinals: dyadic partitions give a map to $2^\N$ with null fibres, and $<\kappa$-additivity forces $\kappa\leqslant\cont$. These observations concern existence assumptions alone. They neither settle the positive assertion under $\RVM(\cont)$ nor rule out a positive consistency result obtained from an additional forcing or continuum axiom.

\section{Comparison with Banach quotients}\label{sec:banach}

The proof of Theorem~C uses the finite-sign cancellation underlying \cite{KSZProducts,SZ}, together with simultaneous avoidance of null sets. This also illustrates the part of the small-density sampling idea in \cite[Lemma~3.2]{OpenAI} that remains available for pointwise function spaces. The latter preprint prompted this comparison; the finite-support probabilistic method is already present in the earlier literature. Its finite-support nature is essential: arbitrary elements of $C(K)^*=M(K)$ need not define continuous functionals on $C_p(K)$.

The final contradiction in the announced Banach argument \cite[Section~3]{OpenAI} cannot be used here. There, absence of a separable Banach quotient rules out unconditional basic sequences in the dual; the argument then works with a separable subspace containing no copy of $\ell_1$. For pointwise spaces, however, distinct points always give
\[
\left\|\sum_{j=1}^m a_j\delta_{x_j}\right\|=\sum_{j=1}^m|a_j|.
\]
Thus the norm completion of their span is $\ell_1$, including when $K$ is our diamond counterexample. The later Borel-coding and Fubini argument in \cite[Lemma~3.4]{OpenAI} depends on the reduction excluding $\ell_1$ and does not transfer by this construction. Producing an unconditional norm-basic sequence supplies no contradiction to failure of a pointwise quotient. The required condition is relative weak-star closedness of a countably infinite-dimensional subspace of $L(K)$, as proved in Lemma~\ref{lem:dual}.

Moreover, an unrestricted positive proof cannot simply produce $(c_0)_p$ in every case. The space $C_p(\beta\N)$ has a quotient isomorphic to $(\ell_\infty)_p$, but has no quotient isomorphic to $(c_0)_p$ \cite[Corollary~2]{BKS2019}. Since $w(\beta\N)=\cont$, this also shows why extending the finite-support Josefson--Nissenzweig conclusion through the boundary $w(K)=\cont$ is impossible in general. The full quotient question permits other targets.

Neither Theorem~C nor the announced Banach-space result settles the unrestricted positive pointwise problem under $\RVM(\cont)$. Theorem~\ref{main:rvm} isolates the remaining case of separable compacta of weight $\cont$. Likewise, the diamond construction supplies no counterexample from CH alone: Proposition~\ref{prop:prediction} identifies the strength of its present prediction step, not the weakest possible hypothesis for a different construction.

\end{document}